\documentclass[11pt,a4paper]{article}
\usepackage{enumerate, amsmath, amsthm}
\usepackage{amssymb, amsfonts, mathrsfs}
\usepackage{comment}
\usepackage{array}
\usepackage{hyperref}
\usepackage{xcolor}
\usepackage{float}
\usepackage{subcaption}
\usepackage{bm}

\usepackage{tikz}
\usetikzlibrary{shapes.geometric, calc, backgrounds}
\tikzstyle{vertex}=[draw, fill=gray, circle, minimum size={0.16cm}, inner sep=0cm, font=\bf, align=center, scale=0.88]
\tikzstyle{vertex}=[circle, fill, scale=0.5]

\usepackage{graphicx}

\theoremstyle{definition}
\newtheorem{thm}{Theorem}

\newtheorem{lem}{Lemma}

\newtheorem{prop}{Proposition}
\usepackage{hyperref}

\bfseries\normalfont

\begin{document}

\title{HIST-Critical Graphs and Malkevitch's Conjecture}

\author{
Jan Goedgebeur\footnote{Department of Computer Science, KU Leuven Kulak, 8500 Kortrijk, Belgium}\;\footnote{Department of Mathematics, Computer Science and Statistics, Ghent University, 9000 Ghent, Belgium}\;,
Kenta Noguchi\thanks{Department of Information Sciences,
Tokyo University of Science,
2641 Yamazaki, Noda, Chiba 278-8510, Japan}\,, Jarne Renders\footnotemark[1]\;, Carol T. Zamfirescu\footnotemark[2]\;\footnote{Department of Mathematics, Babe\c{s}-Bolyai University, Cluj-Napoca, Roumania}\;\footnote{E-mail addresses: jan.goedgebeur@kuleuven.be; noguchi@rs.tus.ac.jp; jarne.renders@kuleuven.be; czamfirescu@gmail.com}
}

\date{}
\maketitle

\noindent
\begin{abstract}
In a given graph, a \textit{HIST} is a spanning tree without 2-valent vertices. Motivated by developing a better understanding of \textit{HIST-free} graphs, i.e.\ graphs containing no HIST, in this article's first part we study \textit{HIST-critical} graphs, i.e.\ HIST-free graphs in which every vertex-deleted subgraph does contain a HIST (e.g.\ a triangle). We give an almost complete characterisation of the orders for which these graphs exist and present an infinite family of planar examples which are 3-connected and in which nearly all vertices are 4-valent. This leads naturally to the second part in which we investigate planar 4-regular graphs with and without HISTs, motivated by a conjecture of Malkevitch, which we computationally verify up to order~$22$. First we enumerate HISTs in antiprisms, whereafter we present planar 4-regular graphs with and without HISTs, obtained via line graphs. Finally, we confirm Malkevitch's conjecture for the family of line graphs of cyclically $4$-edge connected cubic graphs.
\end{abstract}



\section{Introduction}

In a given graph, a spanning tree without 2-valent vertices is called a \emph{HIST}, an abbreviation of \emph{homeomorphically irreducible spanning tree}. A graph not containing a HIST is \emph{HIST-free}. HIST-free graphs play an important role in the theory of these spanning trees, see for instance the work of Albertson, Berman, Hutchinson, and Thomassen~\cite{ABHT90}, and many fundamental questions remain unanswered. We will call a graph $G$ \textit{$K_1$-histonian} 
if every vertex-deleted subgraph of $G$ contains a HIST. In this article our aim is to investigate HIST-freeness from two perspectives: in the first part we focus on \textit{HIST-critical} graphs, i.e.\ HIST-free $K_1$-histonian graphs, e.g.\ $K_3$. In the second part we study Malkevitch's Conjecture stating that planar 4-connected graphs must contain a HIST. We point out that the question whether a graph contains a HIST or not has been intensely investigated, see for instance \cite{CRS12,CS13,HNO18,NT18}.

We recall that a cycle in a graph is \textit{hamiltonian} if it visits every vertex of the graph, and a graph is \textit{hamiltonian} if it contains a hamiltonian cycle. So in a given graph a hamiltonian cycle is a connected spanning subgraph in which \textit{every} vertex has degree~2, while a HIST is a connected spanning subgraph in which \textit{no} vertex has degree~2. Just like the conjecture of Malkevitch~\cite{Ma79} stating that every planar 4-connected graph contains a HIST aims at establishing a HIST-analogue of Tutte's celebrated theorem that planar 4-connected graphs are hamiltonian~\cite{Tu56}, much of the work here is motivated by the desire to better understand in which cases hamiltonian cycles and hamiltonicity-related concepts behave like their HIST counterparts, and in which cases they do not (and, of course, why this is so). We remark that the notion ``HIST-critical'' has a hamiltonian counterpart in ``\textit{hypohamiltonian}'': these are non-hamiltonian graphs in which every vertex-deleted subgraph is hamiltonian. 



Throughout the article we will use a combination of theoretical and computational arguments. Therefore, in Section~\ref{sec:algo} we first present the algorithm we used to test whether or not a graph is HIST-free and to count its number of HISTs if it is not. In Section~\ref{sec:hist-critical} we focus on HIST-critical graphs and give an almost complete characterisation of the orders for which these graphs exist and present an infinite family of planar examples which are 3-connected and in which nearly all vertices are 4-valent. In Section~\ref{sec:malkevitch} we prove a series of results motivated by Malkevitch's Conjecture. We show by computational means that it holds up to at least $22$ vertices; determine the minimum number of HISTs in planar 4-connected graphs on at most $18$ vertices; prove that antiprisms of order $2k$ with $k \geq 3$ have exactly $2k(2k-2)$ HISTs; and show that there exist 4-connected HIST-free graphs of genus 3. 

We shall use the notation $[n] := \{ 0, \ldots, n \}$. Unless specified otherwise, all graphs in this paper
are simple; while \emph{multigraphs} can have parallel edges (but no loops). For a graph $G$, the \emph{degree} of a vertex $v$ is the number of incident edges of $v$, denoted by $d_G(v)$. We call a vertex of degree $k$ a \textit{$k$-vertex}.


\section{Algorithm for counting HISTs}
\label{sec:algo}

We implemented an efficient backtracking algorithm to test whether or not a graph is HIST-free and to count its number of HISTs if any are present. The main idea of our algorithm is a straightforward way of searching for the spanning trees of the graph $G$ by recursively adding edges to a tree $T$ and forbidding edges from being added. These forbidden edges induce a subgraph of $G$, say $G'$. 

Given some subtree $T$ and subgraph $G'$ of our graph, we find a vertex $v$ for which $d_G(v) - d_{G'}(v) - d_T(v)$ is non-zero but minimal and which contains a neighbour $w$ such that $vw\not\in E(G')$ and $vw\not\in E(T)$. Let $w$ be this neighbour for which $d_G(w) - d_{G'}(w) - d_T(w)$ is minimal.
At this point we branch. We add $vw$ to $G'$, forbidding it from being added to the tree in this branch, and recurse. If $w$ does not already belong to $T$, we add $vw$ to $T$ and recurse.

A spanning tree is found if $|E(T)| = |V(G)| - 1$ and it is a HIST if there are no vertices of degree~$2$. We use certain additional elementary pruning criteria. For example, if for a vertex $v$ we have $d_G(v) = d_{G'}(v)$ or $d_G(v) - d_{G'}(v) = d_T(v) = 2$, we can prune the current branch.

Stopping the search once a HIST is found gives us an algorithm for checking whether a graph is HIST-free. This algorithm can then also be used to determine whether a graph is $K_1$-histonian or HIST-critical.


While the correctness of the algorithm can easily be proven, there is always the risk of errors in the implementation of the algorithm. To mitigate these we verified many of our results using an independent implementation of this algorithm as well as the implementation of a different backtracking algorithm, which starts from an initial spanning tree and alters it by adding an removing edges in such a way that we will have encountered all HISTs, similar to the algorithm of Kapoor and Ramesh~\cite{KR95} for spanning trees. For details on the verification, we refer the reader to Appendix~\ref{app:correctness}.

Our implementation of the algorithm is open source software and can be found on GitHub~\cite{GNRZ23} where it can be verified and used by other researchers.

\section{HIST-Critical Graphs}
\label{sec:hist-critical}

In the theory of hamiltonicity, so-called \textit{hypohamiltonian} graphs---non-hamiltonian graphs in which every vertex-deleted subgraph is hamiltonian---play a special role. The smallest such graph is the famous Petersen graph. The investigation of this class of graphs started in the sixties and results have appeared in a steady stream throughout the decades. In the beginning, it seemed that they were closely related to snarks---many snarks being hypohamiltonian graphs and vice-versa---but as more and more examples were described, it became clear that the families are quite different. In line with this early perceived similarity, Chv\'atal~\cite{Ch73} asked whether \textit{planar} hypohamiltonian graphs exist, and Gr\"unbaum~\cite{Gr74} conjectured their non-existence; recall that the non-planarity of snarks is equivalent to the Four Colour Theorem. Thomassen proved, using Grinberg's hamiltonicity criterion as an essential tool, that infinitely many planar hypohamiltonian graphs exist~\cite{Th76}. 

In this section, we look at the HIST-analogue of hypohamiltonian graphs, namely HIST-critical graphs. As mentioned in the introduction, these are HIST-free graphs in which every vertex-deleted subgraph does contain a HIST. We mirror the development in hypohamiltonicity theory in two ways: first we give a near-complete characterisation of orders for which HIST-critical graphs exist (this was completed for hypohamiltonian graphs by Aldred, McKay, and Wormald~\cite{AMW97}), and then show that infinitely many planar HIST-critical graphs exist, paralleling Thomassen's aforementioned result. 

The latter family consists of 3-connected graphs in which nearly all vertices are 4-valent. Our desire to find planar HIST-critical graphs with few cubic vertices is motivated as follows. Thomassen also proved the surprising structural result that every planar hypohamiltonian graph must contain a cubic vertex~\cite{Th78}. Note that this is equivalent to the statement that every planar graph with minimum degree at least 4 and in which every vertex-deleted subgraph is hamiltonian must itself be hamiltonian---this strengthens Tutte's celebrated theorem that planar 4-connected graphs are hamiltonian. It remains unknown whether every planar HIST-critical graph must contain a cubic vertex.

Clearly, $K_3$ is the smallest HIST-critical graph. The first question one can ask, in the spirit of establishing which parallels between HIST-critical and hypohamiltonian graphs hold and which do not, is whether there is a HIST-analogue of the Petersen graph. One key property of the Petersen graph is that it is 3-regular. We now give the easy proof of a fact that makes it impossible to find a suitable HIST-analogue of the Petersen graph.

\begin{prop}
\textit{In a graph of even order and maximum degree at most~$3$, there exists no vertex-deleted subgraph with a HIST. In particular, there are no $3$-regular HIST-critical graphs.}
\end{prop}

\begin{proof}
Any HIST $T$ of a graph $H$ of maximum degree at most 3 contains only 1- and 3-vertices. The difference between the number of 1- and 3-vertices present in $T$ must be 2. So the order of $T$ and thus of $H$ must be even. But the graph $G$ from the statement is required to have even order, so its vertex-deleted subgraphs must have odd order.
\end{proof}

To obtain other examples---in particular, in light of the above observation, examples of \textit{even} order---, we used \texttt{geng}~\cite{MP14} to exhaustively generate general $2$-connected graphs and used our algorithm from Section~\ref{sec:algo} to test which of the generated graphs are HIST-critical. Note that HIST-critical graphs are $2$-connected. The results are summarised in Table~\ref{tab:2-conn_HIST-critical}.
The table shows the existence of several HIST-critical graphs other than $K_3$, but none of even order. See Figure~\ref{fig:5_smallest_HIST-critical} in the Appendix for illustrations of the five smallest HIST-critical graphs. In the hope of finding more examples, we also computed HIST-critical graphs under girth restrictions, as this allowed us to look at higher orders. See Figure~\ref{fig:HIST-critical_girth_restrictions} in the Appendix for a HIST-critical graph with girth $4, 5, 6$ and $7$.
The results can also be found in Table~\ref{tab:2-conn_HIST-critical}. All graphs from this table are available on the \textit{House of Graphs}~\cite{CDG23} at \url{https://houseofgraphs.org/meta-directory/hist-critical}.
Now we did find examples of even order. We shall now prove that there are in fact infinitely many such graphs. 

\begin{table}[H]
    \centering
    \begin{tabular}{c | r | r | r | r | r}
        Order & $h(n)$ & $h(4,n)$ & $h(5,n)$ & $h(6,n)$ & $h(7,n)$\\\hline
        3 & 1 & 0 & 0 & 0 & 0\\
        4, 5, 6 & 0 & 0 & 0 & 0 & 0\\
        7 & 2 & 0 & 0 & 0 & 0\\
        8 & 0 & 0 & 0 & 0 & 0\\
        9 & 2 & 0 & 0 & 0 & 0\\
        10 & 0 & 0 & 0 & 0 & 0\\
        11 & 35 & 3 & 1 & 0 & 0\\
        12 & 0 & 0 & 0 & 0 & 0\\
        13 & 153 & 6 & 2 & 0 & 0\\
        14 & ? & 1 & 1 & 0 & 0\\
        15 & ? & 149 & 25 & 0 & 0\\
        16 & ? & 3 & 0 & 0 & 0\\
        17 & ? & ? & 244 & 0 & 0\\
        18 & ? & ? & 1 & 0 & 0\\
        19 & ? & ? & 4\,129 & 4 & 0\\
        20 & ? & ? & 3 & 1 & 0\\
        21 & ? & ? & ? & 98 & 0\\
        22 & ? & ? & ? & 0 & 0\\
        23 & ? & ? & ? & 6\,036 & 0\\
        24 & ? & ? & ? & 52 & 0\\
        25, 26 & ? & ? & ? & ? & 0\\
        27 & ? & ? & ? & ? & 8\\
    \end{tabular}
    \caption{Exact counts of HIST-critical graphs with a given lower bound on the girth. Column $h(k,n)$ gives the number of HIST-critical graphs on $n$ vertices and with girth at least $k$. We put $h(n) := h(3,n)$.}
    \label{tab:2-conn_HIST-critical}

\end{table}

During our search, we noticed that, given a cubic graph (with girth restrictions), one sometimes obtains a HIST-critical graph by subdividing the proper edges. See for example the graph of Figure~\ref{fig:pappus} in the Appendix. It is a HIST-critical graph of girth $6$ obtained by subdividing the Pappus graph in three places. 

Using this observation and starting from cubic graphs of girth equal to $8$ or $9$, we were able to find HIST-critical graphs of girth $8$ of order $37, 39$ and $41$, and of girth $9$ of order $59$ in a non-exhaustive way. An example of such a girth $8$ graph of order $41$ is available at \url{https://houseofgraphs.org/graphs/50549} and an example of such a girth $9$ graph of order $59$ is available at \url{https://houseofgraphs.org/graphs/50547}. Their existence will be used in the proof of Theorem~\ref{orders}.

\subsection{HIST-Critical fragments}
\label{sec:hist-crit-fragments}



Ultimately, our goal is a HIST-analogue of the result of Aldred, McKay, and Wormald~\cite{AMW97} stating that there exists a hypohamiltonian graph of order $n$ if and only if $n \in \{ 10, 13, 15, 16 \}$ or $n \ge 18$. Although our characterisation given in Section~\ref{sec:charcterisation} is not complete, only few orders shall remain open.

Let $F$ be a graph with $V(F) = \{x, y, v_1, \ldots, v_\ell\}$ where $\ell \ge 1$. For a given connected subgraph $H$ of $F$, we call a spanning tree (spanning forest) $\Upsilon$ of $H$ an \emph{$\{ x,y \}$-excluded HIST} (\emph{$\{ x,y \}$-excluded HISF)} if $d_\Upsilon(v) \ne 2$ for all $v \in V(H) \setminus \{ x, y \}$. A graph $F$ will be called a \textit{HIST-critical $\{ x,y \}$-fragment} if it satisfies all of the following properties. 
\begin{enumerate}
	\item
	$F$ has an $\{ x,y \}$-excluded HIST. Moreover, every $\{ x,y \}$-excluded HIST $T$ of $F$ satisfies $d_T(x) = d_T(y) = 2$; 
	\item \label{prop:vertex_deleted_have_xy_excl_HIST}
	$F - x$ has an $\{ x,y \}$-excluded HIST with $d_F(y) \ne 1$ and $F - y$ has an $\{ x,y \}$-excluded HIST with $d_F(x) \ne 1$;
	\item
	for every $v \in V(F) \setminus \{ x, y \}$, the graph $F - v$ either has
	(a) an $\{ x,y \}$-excluded HIST with at least one of $x$ and $y$ of degree $\ne 2$, or 
	(b) an $\{ x,y \}$-excluded HISF consisting of exactly two components $T_x$ and $T_y$, each on at least two vertices,
    such that $x \in V(T_x)$ and $y \in V(T_y)$; and
	\item
	$F$ does not have an $\{ x,y \}$-excluded HISF with property (3b) above. 
\end{enumerate}



\begin{thm}\label{fragment-chain}
\textit{Let $k \ge 2$ be an integer. For all $i \in [k - 1]$, consider pairwise disjoint HIST-critical $\{ x_i,y_i \}$-fragments $H_i$, and identify $y_i$ with $x_{i+1}$, indices mod~$k$. The resulting graph $G$ is HIST-critical.}
\end{thm}

\begin{proof} 
In this proof we see $H_i$ as a subgraph of $G$ for every $i \in [k-1]$. We first show that $G$ is $K_1$-histonian. By (1), for every $i \in [k-1]$ the graph $H_i$ contains an $\{ x_i,y_i \}$-excluded HIST $T_i$ satisfying $d_{T_i}(x_i) = d_{T_i}(y_i) = 2$. By (2), $H_0 - x_0$ has an $\{ x_0,y_0 \}$-excluded HIST $T'_0$ with $d_{T_0}(y_0) \ne 1$ and $H_{k-1} - y_{k-1}$ has an $\{ x_{k-1},y_{k-1} \}$-excluded HIST $T'_{k-1}$ with $d_{T_{k-1}}(x_{k-1}) \ne 1$. Then, since $x_0 = y_{k-1}$, the tree $T'_0 \cup \bigcup_{i=1}^{k-2} T_i \cup T'_{k-1}$ is a HIST of $G - x_0$. Finding a HIST of $G - v$ is analogous for any other vertex $v\in \{ x_i, y_i \}_{i \in [k-1]}$.

Consider $v \in V(H_0) \setminus \{ x_0, y_0 \}$. By (3), $H_0 - v$ either has
(a) an $\{ x_0,y_0 \}$-excluded HIST $S$ with at least one of $x_0$ and $y_0$ of degree $\ne 2$, or (b) an $\{ x_0,y_0 \}$-excluded HISF consisting of exactly two components $S_{x_0}$ and $S_{y_0}$, each on at least two vertices, such that $x_0 \in V(S_{x_0})$ and $y_0 \in V(S_{y_0})$. We first treat case~(a). We may assume without loss of generality $d_{S}(x_0) \ne 2$. Then $S \cup \bigcup_{i=1}^{k-2} T_i \cup T'_{k-1}$ is a HIST of $G - v$, where $T_1, \ldots, T_{k-2}, T'_{k-1}$ are defined as in the preceding paragraph. For case~(b), the tree $S_{y_0} \cup \bigcup_{i=1}^{k-1} T_i \cup S_{x_0}$ is a HIST of $G - v$.

We now show that $G$ is HIST-free. Assume $G$ does contain a HIST $T$. Put $T_i := T \cap H_i$. A \textit{HISF} shall be a disjoint union of HISTs. By construction and in particular by (1) (we should rule out why $\bigcup_{i=0}^{k-2} T_i \cup T'_{k-1}$ is not a HIST) there exists a $j \in [k-1]$ such that $T_i$ is a HIST for all $i \in [k-1] \setminus \{ j \}$ and $T_j$ is a HISF consisting of exactly two components, one containing $x_j$, the other containing $y_j$. A priori, one of these components might be isomorphic to $K_1$, but this is in fact impossible: every HIST of $T_i$ is an $\{ x_i, y_i \}$-excluded HIST of $T_i$, so by (1) the $T_i$-degrees of $x_i$ and $y_i$ must be 2, so single-vertex components in the aforementioned HISF cannot occur because this would signify the presence of a 2-vertex in $T$. Every HISF of $T_j$ is also an $\{ x_j, y_j \}$-excluded HISF of $T_j$. But now the existence of $T_j$ contradicts (4), so $G$ is HIST-free.
\end{proof}

First, we remark that the degree requirements on $x$ and $y$ in property~(\ref{prop:vertex_deleted_have_xy_excl_HIST}) are only necessary when $k=2$. 

Let $F_1$ and $F_2$ be graphs defined as follows; see also Figures~\ref{fig:3-1} and~\ref{fig:3-2} in the Appendix. Note that $F_1$ is the Petersen graph from which two adjacent vertices were removed.	
\begin{eqnarray*}
V(F_1) &=& \{ x, y, v_1, \ldots, v_6 \} \\
E(F_1) &=& \{ xv_3, xv_6, yv_1, yv_4, v_1v_2, v_1v_6, v_2v_3, v_3v_4, v_4v_5, v_5v_6 \} \\
V(F_2) &=& \{ x, y, v_1, \ldots, v_{10} \} \\
E(F_2) &=& \{ xv_1, xv_8, yv_6, yv_9, v_1v_2, v_1v_6, v_1v_7, v_2v_3, v_3v_4, v_3v_8, v_4v_5, v_4v_9, v_5v_6, v_6v_{10}, v_7v_9, v_8v_{10} \}
\end{eqnarray*}

\begin{prop}\label{fragments}
\textit{The graphs $F_1$ 
and $F_2$ 
are HIST-critical $\{ x,y \}$-fragments.}
\end{prop}

\begin{proof}
For $F_1$ and $F_2$, properties (2) and (3) can be checked by Figures~\ref{fig:3-1} and~\ref{fig:3-2} in the Appendix, using symmetry. 
Now we confirm properties (1) and (4). 

For $F_1$, let $\Upsilon$ be either an $\{ x,y \}$-excluded HIST or an $\{ x,y \}$-excluded HISF with property (3b). In $\Upsilon$, precisely one of the two edges $v_1v_2$ or $v_2v_3$ is used. If $v_1v_2 \in E(\Upsilon)$, then $yv_1, v_1v_6 \in E(\Upsilon)$. Similarly, precisely one of $v_4v_5$ or $v_5v_6$ is present. On the one hand, if $v_5v_6 \in E(\Upsilon)$, then $xv_6, v_1v_6 \in E(\Upsilon)$ and we see that $\Upsilon$ is the $\{ x,y \}$-excluded HIST with $E(\Upsilon) = \{ xv_3, xv_6, yv_1,yv_4, v_1v_2, v_1v_6, v_5v_6 \}$. (See the top centre drawing of Figure~\ref{fig:3-1}.) On the other hand, if $v_4v_5 \in E(\Upsilon)$, then $yv_4, v_3v_4 \in E(\Upsilon)$ and, hence, $x$ cannot be in $\Upsilon$ which is a contradiction. 
If the edge $v_2v_3 \in E(\Upsilon)$, then the same argument implies that $\Upsilon$ is the $\{ x,y \}$-excluded HIST with $E(\Upsilon) = \{ xv_3, xv_6, yv_1,yv_4, v_2v_3, v_3v_4, v_4v_5 \}$. In both cases, $\Upsilon$ is an $\{ x,y \}$-excluded HIST with $d_{\Upsilon}(x) = d_{\Upsilon}(y) = 2$.

For $F_2$, let $\Upsilon$ be either an $\{ x,y \}$-excluded HIST or an $\{ x,y \}$-excluded HISF with the property (3b). In $\Upsilon$, precisely one of the two edges $v_1v_2$ or $v_2v_3$ is used and precisely one of the two edges $v_4v_5$ or $v_5v_6$ is used. We consider the three cases by symmetry: $v_1v_2, v_4v_5 \in E(\Upsilon)$, $v_1v_2, v_5v_6 \in E(\Upsilon)$, and $v_2v_3, v_4v_5 \in E(\Upsilon)$. 
If $v_1v_2, v_4v_5 \in E(\Upsilon)$, then $v_3v_4, v_4v_9 \in E(\Upsilon)$ and, hence, $yv_9, v_7v_9 \in E(\Upsilon)$. To include the two vertices $v_8$ and $v_{10}$, we see that $\Upsilon$ is the $\{ x,y \}$-excluded HIST with $E(\Upsilon) = \{ xv_1, xv_8, yv_6,yv_9, v_1v_2, v_1v_6, v_3v_4, v_4v_5, v_4v_9, v_6v_{10}, v_7v_9 \}$. 
If $v_1v_2, v_5v_6 \in E(\Upsilon)$, then $v_3v_8, v_4v_9 \in E(\Upsilon)$ to include the two vertices $v_3$ and $v_4$ and, hence, $xv_8, yv_9, v_7v_9, v_8v_{10} \in E(\Upsilon)$. Then we see that $\Upsilon$ is the $\{ x,y \}$-excluded HIST with $E(\Upsilon) = \{ xv_1, xv_8, yv_6,yv_9, v_1v_2, v_1v_6, v_3v_8, v_4v_9, v_5v_6, v_7v_9, v_8v_{10} \}$. (See the top centre drawing of Figure~\ref{fig:3-2}.) 
If $v_2v_3, v_4v_5 \in E(\Upsilon)$, then $v_3v_4, v_3v_8, v_4v_9 \in E(\Upsilon)$ and, hence, by symmetry we can assume that $xv_8, v_8v_{10} \in E(\Upsilon)$. If $yv_9 \not\in E(\Upsilon)$, then $yv_6 \in E(\Upsilon)$ and we see that $v_7$ cannot be in $\Upsilon$. So we have $yv_9 \in E(\Upsilon)$ and we see that $\Upsilon$ is the $\{ x,y \}$-excluded HIST with $E(\Upsilon) = \{ xv_1, xv_8, yv_6,yv_9, v_2v_3, v_3v_4, v_3v_8, v_4v_5, v_4v_9, v_7v_9, v_8v_{10} \}$. 
In all cases, $\Upsilon$ is an $\{ x,y \}$-excluded HIST with $d_{\Upsilon}(x) = d_{\Upsilon}(y) = 2$. 
%
\end{proof}

\subsection{Gluing \texorpdfstring{\boldmath$K_1$}{K1}-histonian graphs}

One might wonder whether gluing procedures---which are very successful in the context of hypohamiltonian graphs---can be formulated for HIST-critical graphs. Unfortunately, the next observation shows that a very natural gluing procedure applied to two $K_1$-histonian graphs (this includes all HIST-critical graphs) always yields a graph containing a HIST.

Let $G$ and $H$ be disjoint graphs. We consider non-adjacent vertices $x_G,y_G$ in $G$ and non-adjacent vertices $x_H,y_H$ in $H$. First, identify $x_G$ with $x_H$ and $y_G$ with $y_H$; the obtained vertices will be denoted by $x$ and $y$. Thereafter, add a new vertex $z$ and join it to $x$ and $y$. Finally, add the edge $xy$. The resulting graph shall be denoted by $(G,x_G,y_G):(H,x_H,y_H)$, and when the choice of $x_G,y_G$ and $x_H,y_H$ plays no role, we simply write $G:H$. Observe that this can be seen as identifying two non-adjacent vertices in two $K_1$-histonian graphs, and then identifying these two vertices with two vertices of a triangle, which is HIST-critical.

\begin{prop}\label{prop1}
	\emph{If $G$ and $H$ are $K_1$-histonian, then $G:H$ is $K_1$-histonian. Moreover, $G:H$ contains a HIST.}
\end{prop}

\begin{proof}
	Throughout the proof we see $G$ and $H$ as subgraphs of $\Gamma := G:H$. In $\Gamma - x$, we obtain a HIST $T$ by taking the union of the HIST present in $G - x$, the HIST present in $H - x$, and $(\{ y, z \}, yz)$, thus guaranteeing that the degree of $y$ in $T$ is at least 3. Analogously we obtain a HIST in $\Gamma - y$. In $\Gamma - z$, consider $T - yz + xy$. By considering $T + xy$, we see that $\Gamma$ itself must contain a HIST.
	
	Now let $v$ be a vertex in $G - x - y$. Consider a HIST $T_G^v$ in $G - v$ and a HIST $T_H^x$ in $H - x$. Then $T_G^v \cup T_H^x \cup (\{ y, z \}, yz)$ is the desired HIST in $\Gamma - v$. For a vertex in $H - x - y$ we can use the same argument, thus completing the proof.
\end{proof}

\subsection{Planar HIST-critical graphs}

Here we give an exhaustive list of the counts of all planar HIST-critical graphs up to order 14, and present an infinite family of planar HIST-critical graphs.

Using \texttt{plantri}~\cite{BM07} we generated all planar $2$-connected graphs up to order $14$ and used our algorithm from Section~\ref{sec:algo} to determine which are HIST-critical. The results can be found in Table~\ref{tab:counts_hist-crit_2conn_planar}.
All graphs from this table can be obtained from the \textit{House of Graphs}~\cite{CDG23} at \url{https://houseofgraphs.org/meta-directory/hist-critical} and also be inspected in the database of interesting graphs at the House of Graphs by searching for the keywords ``planar HIST-critical''.


\begin{table}[H]
\centering
\begin{tabular}{c | r r r r r r r r r r r r }
Order & 3 & 4 & 5 & 6 & 7 & 8 & 9 & 10 & 11 & 12 & 13 & 14 \\ \hline
HIST-critical & 1 & 0 & 0 & 0 & 2 & 0 & 0 & 0 & 12 & 0 & 12 & 0 \
\end{tabular}
\caption{Exact counts of planar HIST-critical graphs for each order.} 
\label{tab:counts_hist-crit_2conn_planar}
\end{table}

Motivated by corresponding problems for hypohamiltonian graphs (as described at the beginning of this section), we shall now present an infinite family of HIST-critical graphs, with the added property of planarity. The second part of the next theorem is inspired by a similar question of Chv\'atal on hypohamiltonian graphs, see~\cite{Ch73,Th74,Za15}. 

\begin{thm}\label{thm1}
	\emph{There are infinitely many planar HIST-critical graphs. Moreover, there exist infinitely many planar HIST-critical graphs $G$, each containing an edge $e$ such that $G - e$ is $3$-connected and HIST-critical.}
\end{thm}

\begin{proof}
	For each integer $k \ge 3$, let $G_k$ be a planar graph with vertex set and edge set defined as follows.
\begin{eqnarray*}
V(G_k) &=& \{a_1, \ldots, a_k, b_1, \ldots, b_k, c_1, \ldots, c_{k-1}, x, y\} \\
E(G_k) &=& \{a_ia_{i+1}, a_ic_i, a_{i+1}c_i, b_ib_{i+1}, b_ic_i, b_{i+1}c_i \mid 1 \le i \le k-1 \} \cup \{a_1x, a_kx, b_1y, b_ky, xy \}
\end{eqnarray*}
	Its plane embedding is depicted in Figure~\ref{subfig:a}, where $x$ and $y$ are adjacent. 
(It is not difficult to see that $G_k$ is 3-connected,
and this embedding is unique by a classic result of Whitney.)
	

	We show that, for every even integer $k \ge 4$, both the graph $G_k$ and the graph $H_k := G_k + a_1a_k$ are planar HIST-critical graphs. 
Thus, $H_k$ is the desired infinite family. 
First, we prove that $H_k$ is HIST-free. 
Suppose $H_k$ does have a HIST $T$. Since $|V(T)| = 3k+1$ is odd, $H$ has a vertex of even degree by the degree sum formula, that is, a $4$-vertex $v$. Note that every $4$-vertex in $H_k$ lies on two adjacent triangles $vpq$ and $vrs$. Thus, exactly four edges, namely $vp, vq, vr, vs$ in the two triangles lie in $T$. Since $T$ spans all vertices of $H_k$, at least one of $p, q, r, s$ should be of degree~$3$ in $T$, say, $p$. Then $p$ should have degree~$4$ in $H_k$ and $p$ lies on another triangle, say, $ptu$, and so exactly two edges $pt, pu$ in the triangle must be in $T$. By this argument, it is easy to see that $T$ contains none of the three edges $b_1y, b_ky, xy$ which do not lie on a triangle, and hence $T$ does not contain the vertex $y$, a contradiction. Hence, $G_k$ is HIST-free, too.
	
	Next, we show that for every $v \in V(G_k)$, $G_k - v$ has a HIST. By symmetry, we only need to consider the following five cases.
	\begin{itemize}
	\item For $v = x$, see Figure~\ref{subfig:b}.
	\item
     For $v =a_i$ where $i \in \{1, 3, \ldots, k-3\}$, see Figure~\ref{subfig:c}. 
One should add edges $$a_1c_1, a_2c_1, a_3c_3, a_4c_3, \ldots, a_{i-2}c_{i-2}, a_{i-1}c_{i-2}, a_{i+1}c_{i+1}, a_{i+2}c_{i+1}, \ldots, a_{k-4}c_{k-4}, a_{k-3}c_{k-4}.$$ 
	\item For $v = a_{k-1}$, see Figure~\ref{subfig:d}.
	\item For $v \in \{c_1, c_3, \ldots, c_{k-1}\}$, see Figure~\ref{subfig:e}.
	\item For $v \in \{c_2, c_4, \ldots, c_{k-2}\}$, see Figure~\ref{subfig:f}.
	\end{itemize}
	
	It follows that for every $v \in V(H_k)$, $H_k$ has a HIST, too.
	
\end{proof}

\begin{figure}[!htb]
    \centering
    \newcommand{\x}{1.1}
    \newcommand{\y}{1.3}
    \begin{subfigure}[b]{0.32\textwidth}
    \centering
    \begin{tikzpicture}[scale = 0.5]
        \node[vertex] (0) at (-1*\x, 3*\y) {};
        \node[vertex] (1) at (-1*\x, 2*\y) {};
        \node[vertex] (2) at (-1*\x, 1*\y) {};
        \node[vertex] (3) at (-1*\x, -1*\y) {};
        \node[vertex] (4) at (-1*\x, -2*\y) {};
        \node[vertex] (5) at (-1*\x, -3*\y) {};
        \node[vertex] (6) at (0*\x, 2.5*\y) {};
        \node[vertex] (7) at (0*\x, 1.5*\y) {};
        \node[vertex] (8) at (0*\x, -1.5*\y) {};
        \node[vertex] (9) at (0*\x, -2.5*\y) {};
        \node[vertex] (10) at (1*\x, 3*\y) {};
        \node[vertex] (11) at (1*\x, 2*\y) {};
        \node[vertex] (12) at (1*\x, 1*\y) {};
        \node[vertex] (13) at (1*\x, -1*\y) {};
        \node[vertex] (14) at (1*\x, -2*\y) {};
        \node[vertex] (15) at (1*\x, -3*\y) {};
        \node[vertex] (16) at (-2*\x, 0*\y) {};
        \node[vertex] (17) at (2*\x, 0*\y) {};

         \path[draw]
        (0) edge node {} (2) 
        (3) edge node {} (5) 
        (0) edge node {} (16) 
        (5) edge node {} (16) 
        (10) edge node {} (12) 
        (13) edge node {} (15) 
        (10) edge node {} (17) 
        (15) edge node {} (17) 
        (0) edge node {} (11) 
        (1) edge node {} (12) 
        (3) edge node {} (14) 
        (4) edge node {} (15) 
        (1) edge node {} (10) 
        (2) edge node {} (11) 
        (4) edge node {} (13) 
        (5) edge node {} (14) 
        ;
        \draw (16) -- ($(16) + (-1,0)$);
        \draw (17) -- ($(17) + (1,0)$);

        \draw[very thick,loosely dotted] (0, 0.5)--(0, -0.5);
        \coordinate [label=left: {$a_1$}] () at (0);
        \coordinate [label=above left: {$a_2$}] () at (1);
        \coordinate [label=left: {$a_k$}] () at (5);
        \coordinate [label=right: {$b_1$}] () at (10);
        \coordinate [label=above right: {$b_2$}] () at (11);
        \coordinate [label=right: {$b_k$}] () at (15);
        \coordinate [label=above: {$c_1$}] () at (6);
        \coordinate [label=above: {$c_2$}] () at (7);
        \coordinate [label={[label distance=1.2]below: {$c_{k-1}$}}] () at (9);
        \coordinate [label=above left: {$x$}] () at (16);
        \coordinate [label=above right: {$y$}] () at (17);

    \end{tikzpicture}
    \caption{$G_k$}\label{subfig:a}
    \end{subfigure}
    \begin{subfigure}[b]{0.32\textwidth}
    \centering
    \begin{tikzpicture}[scale = 0.5]
        \node[vertex] (0) at (-1*\x, 3*\y) {};
        \node[vertex] (1) at (-1*\x, 2*\y) {};
        \node[vertex] (2) at (-1*\x, 1*\y) {};
        \node[vertex] (3) at (-1*\x, -1*\y) {};
        \node[vertex] (4) at (-1*\x, -2*\y) {};
        \node[vertex] (5) at (-1*\x, -3*\y) {};
        \node[vertex] (6) at (0*\x, 2.5*\y) {};
        \node[vertex] (7) at (0*\x, 1.5*\y) {};
        \node[vertex] (8) at (0*\x, -1.5*\y) {};
        \node[vertex] (9) at (0*\x, -2.5*\y) {};
        \node[vertex] (10) at (1*\x, 3*\y) {};
        \node[vertex] (11) at (1*\x, 2*\y) {};
        \node[vertex] (12) at (1*\x, 1*\y) {};
        \node[vertex] (13) at (1*\x, -1*\y) {};
        \node[vertex] (14) at (1*\x, -2*\y) {};
        \node[vertex] (15) at (1*\x, -3*\y) {};
        \node[vertex] (16) at (-2*\x, 0*\y) {};
        \node[vertex] (17) at (2*\x, 0*\y) {};

         \path[draw]
        (10) edge node {} (12) 
        (13) edge node {} (15) 
        (10) edge node {} (17) 
        (0) edge node {} (6) 
        (4) edge node {} (9) 
        (1) edge node {} (10) 
        (7) edge node {} (11) 
        (8) edge node {} (13) 
        (5) edge node {} (14) 
        ;
        \draw[white] (16) -- ($(16) + (-1,0)$);
        \draw[white] (17) -- ($(17) + (1,0)$);

        \draw[very thick,loosely dotted] (0, 0.5)--(0, -0.5);
        \coordinate [label=left: {$a_1$}] () at (0);
        \coordinate [label=above left: {$a_2$}] () at (1);
        \coordinate [label=left: {$a_k$}] () at (5);
        \coordinate [label=right: {$b_1$}] () at (10);
        \coordinate [label=above right: {$b_2$}] () at (11);
        \coordinate [label=right: {$b_k$}] () at (15);
        \coordinate [label=above: {$c_1$}] () at (6);
        \coordinate [label=above: {$c_2$}] () at (7);
        \coordinate [label={[label distance=1.2]below: {$c_{k-1}$}}] () at (9);
        \coordinate [label=above left: {$x$}] () at (16);
        \coordinate [label=above right: {$y$}] () at (17);

    \end{tikzpicture}
    \caption{$v=x$}\label{subfig:b}
    \end{subfigure}
    \begin{subfigure}[b]{0.32\textwidth}
    \centering
    \begin{tikzpicture}[scale = 0.5]
        \node[vertex] (0) at (-1*\x, 3*\y) {};
        \node[vertex] (1) at (-1*\x, 2*\y) {};
        \node[vertex] (2) at (-1*\x, 1*\y) {};
        \node[vertex] (3) at (-1*\x, -1*\y) {};
        \node[vertex] (4) at (-1*\x, -2*\y) {};
        \node[vertex] (5) at (-1*\x, -3*\y) {};
        \node[vertex] (6) at (0*\x, 2.5*\y) {};
        \node[vertex] (7) at (0*\x, 1.5*\y) {};
        \node[vertex] (8) at (0*\x, -1.5*\y) {};
        \node[vertex] (9) at (0*\x, -2.5*\y) {};
        \node[vertex] (10) at (1*\x, 3*\y) {};
        \node[vertex] (11) at (1*\x, 2*\y) {};
        \node[vertex] (12) at (1*\x, 1*\y) {};
        \node[vertex] (13) at (1*\x, -1*\y) {};
        \node[vertex] (14) at (1*\x, -2*\y) {};
        \node[vertex] (15) at (1*\x, -3*\y) {};
        \node[vertex] (16) at (-2*\x, 0*\y) {};
        \node[vertex] (17) at (2*\x, 0*\y) {};

         \path[draw]
        (4) edge node {} (5) 
        (10) edge node {} (12) 
        (13) edge node {} (14) 
        (10) edge node {} (17) 
        (15) edge node {} (17) 
        (3) edge node {} (8) 
        (4) edge node {} (9) 
        (6) edge node {} (10) 
        (7) edge node {} (11) 
        (4) edge node {} (13) 
        ;
        \draw (16) -- ($(16) + (-1,0)$);
        \draw (17) -- ($(17) + (1,0)$);

        \draw[very thick,loosely dotted] (0, 0.5)--(0, -0.5);
        \coordinate [label=left: {$a_1$}] () at (0);
        \coordinate [label=above left: {$a_2$}] () at (1);
        \coordinate [label=left: {$a_k$}] () at (5);
        \coordinate [label=right: {$b_1$}] () at (10);
        \coordinate [label=above right: {$b_2$}] () at (11);
        \coordinate [label=right: {$b_k$}] () at (15);
        \coordinate [label=above: {$c_1$}] () at (6);
        \coordinate [label=above: {$c_2$}] () at (7);
        \coordinate [label={[label distance=1.2]below: {$c_{k-1}$}}] () at (9);
        \coordinate [label=above left: {$x$}] () at (16);
        \coordinate [label=above right: {$y$}] () at (17);

    \end{tikzpicture}
    \caption{$v\in\{a_1,a_3,\ldots, a_{k-3}\}$}\label{subfig:c}
    \end{subfigure}

    \bigskip

    \begin{subfigure}[b]{0.32\textwidth}
    \centering
    \begin{tikzpicture}[scale = 0.5]
        \node[vertex] (0) at (-1*\x, 3*\y) {};
        \node[vertex] (1) at (-1*\x, 2*\y) {};
        \node[vertex] (2) at (-1*\x, 0*\y) {};
        \node[vertex] (3) at (-1*\x, -1*\y) {};
        \node[vertex] (4) at (-1*\x, -2*\y) {};
        \node[vertex] (5) at (-1*\x, -3*\y) {};
        \node[vertex] (6) at (0*\x, 2.5*\y) {};
        \node[vertex] (7) at (0*\x, -0.5*\y) {};
        \node[vertex] (8) at (0*\x, -1.5*\y) {};
        \node[vertex] (9) at (0*\x, -2.5*\y) {};
        \node[vertex] (10) at (1*\x, 3*\y) {};
        \node[vertex] (11) at (1*\x, 1*\y) {};
        \node[vertex] (12) at (1*\x, 0*\y) {};
        \node[vertex] (13) at (1*\x, -1*\y) {};
        \node[vertex] (14) at (1*\x, -2*\y) {};
        \node[vertex] (15) at (1*\x, -3*\y) {};
        \node[vertex] (16) at (-2*\x, 0*\y) {};
        \node[vertex] (17) at (2*\x, 0*\y) {};
        \node[vertex] (18) at (0*\x, 0.5*\y) {};

         \path[draw]
        (0) edge node {} (1) 
        (0) edge node {} (16) 
        (5) edge node {} (16) 
        (11) edge node {} (15) 
        (10) edge node {} (17) 
        (15) edge node {} (17) 
        (0) edge node {} (6) 
        (18) edge node {} (12) 
        (2) edge node {} (13)
        (8) edge node {} (14) 
        (9) edge node {} (15) 
        (3) edge node {} (7) 
        ;
        \draw (16) -- ($(16) + (-1,0)$);
        \draw (17) -- ($(17) + (1,0)$);

        \draw[very thick,loosely dotted] (0, 2)--(0, 1);
        \coordinate [label=left: {$a_1$}] () at (0);
        \coordinate [label=above left: {$a_2$}] () at (1);
        \coordinate [label=left: {$a_k$}] () at (5);
        \coordinate [label=right: {$b_1$}] () at (10);
        \coordinate [label=right: {$b_2$}] () at (11);
        \coordinate [label=right: {$b_k$}] () at (15);
        \coordinate [label=above: {$c_1$}] () at (6);
        \coordinate [label=above: {$c_2$}] () at (7);
        \coordinate [label={[label distance=1.2]below: {$c_{k-1}$}}] () at (9);
        \coordinate [label=above left: {$x$}] () at (16);
        \coordinate [label=above right: {$y$}] () at (17);

    \end{tikzpicture}
    \caption{$v=a_{k-1}$}\label{subfig:d}
    \end{subfigure}
    \begin{subfigure}[b]{0.32\textwidth}
    \centering
    \begin{tikzpicture}[scale = 0.5]
        \node[vertex] (0) at (-1*\x, 3*\y) {};
        \node[vertex] (1) at (-1*\x, 2*\y) {};
        \node[vertex] (2) at (-1*\x, 1*\y) {};
        \node[vertex] (3) at (-1*\x, -1*\y) {};
        \node[vertex] (4) at (-1*\x, -2*\y) {};
        \node[vertex] (5) at (-1*\x, -3*\y) {};
        \node[vertex] (6) at (0*\x, 2.5*\y) {};
        \node[vertex] (7) at (0*\x, 1.5*\y) {};
        \node[vertex] (8) at (0*\x, -1.5*\y) {};
        \node[vertex] (9) at (0*\x, -2.5*\y) {};
        \node[vertex] (10) at (1*\x, 3*\y) {};
        \node[vertex] (11) at (1*\x, 2*\y) {};
        \node[vertex] (12) at (1*\x, 1*\y) {};
        \node[vertex] (13) at (1*\x, -1*\y) {};
        \node[vertex] (14) at (1*\x, -2*\y) {};
        \node[vertex] (15) at (1*\x, -3*\y) {};
        \node[vertex] (16) at (-2*\x, 0*\y) {};
        \node[vertex] (17) at (2*\x, 0*\y) {};

         \path[draw]
        (0) edge node {} (2) 
        (0) edge node {} (16) 
        (5) edge node {} (16) 
        (13) edge node {} (15) 
        (10) edge node {} (17) 
        (15) edge node {} (17) 
        (0) edge node {} (6) 
        (1) edge node {} (12) 
        (3) edge node {} (14) 
        (9) edge node {} (15) 
        (7) edge node {} (11) 
        (4) edge node {} (8) 
        ;
        \draw (16) -- ($(16) + (-1,0)$);
        \draw (17) -- ($(17) + (1,0)$);

        \draw[very thick,loosely dotted] (0, 0.5)--(0, -0.5);
        \coordinate [label=left: {$a_1$}] () at (0);
        \coordinate [label=above left: {$a_2$}] () at (1);
        \coordinate [label=left: {$a_k$}] () at (5);
        \coordinate [label=right: {$b_1$}] () at (10);
        \coordinate [label=above right: {$b_2$}] () at (11);
        \coordinate [label=right: {$b_k$}] () at (15);
        \coordinate [label=above: {$c_1$}] () at (6);
        \coordinate [label=above: {$c_2$}] () at (7);
        \coordinate [label={[label distance=1.2]below: {$c_{k-1}$}}] () at (9);
        \coordinate [label=above left: {$x$}] () at (16);
        \coordinate [label=above right: {$y$}] () at (17);

    \end{tikzpicture}
    \caption{$v\in\{c_1,c_3,\ldots, c_{k-1}\}$}\label{subfig:e}
    \end{subfigure}
    \begin{subfigure}[b]{0.32\textwidth}
    \centering
    \begin{tikzpicture}[scale = 0.5]
        \node[vertex] (0) at (-1*\x, 3*\y) {};
        \node[vertex] (1) at (-1*\x, 2*\y) {};
        \node[vertex] (2) at (-1*\x, 1*\y) {};
        \node[vertex] (3) at (-1*\x, -1*\y) {};
        \node[vertex] (4) at (-1*\x, -2*\y) {};
        \node[vertex] (5) at (-1*\x, -3*\y) {};
        \node[vertex] (6) at (0*\x, 2.5*\y) {};
        \node[vertex] (7) at (0*\x, 1.5*\y) {};
        \node[vertex] (8) at (0*\x, -1.5*\y) {};
        \node[vertex] (9) at (0*\x, -2.5*\y) {};
        \node[vertex] (10) at (1*\x, 3*\y) {};
        \node[vertex] (11) at (1*\x, 2*\y) {};
        \node[vertex] (12) at (1*\x, 1*\y) {};
        \node[vertex] (13) at (1*\x, -1*\y) {};
        \node[vertex] (14) at (1*\x, -2*\y) {};
        \node[vertex] (15) at (1*\x, -3*\y) {};
        \node[vertex] (16) at (-2*\x, 0*\y) {};
        \node[vertex] (17) at (2*\x, 0*\y) {};

         \path[draw]
        (0) edge node {} (2) 
        (3) edge node {} (5) 
        (0) edge node {} (16) 
        (5) edge node {} (16) 
        (0) edge node {} (11) 
        (1) edge node {} (7) 
        (9) edge node {} (15) 
        (6) edge node {} (10) 
        (4) edge node {} (8) 
        (5) edge node {} (14) 
        ;
        \draw (16) -- ($(16) + (-1,0)$);
        \draw (17) -- ($(17) + (1,0)$);

        \draw[very thick,loosely dotted] (0, 0.5)--(0, -0.5);
        \coordinate [label=left: {$a_1$}] () at (0);
        \coordinate [label=above left: {$a_2$}] () at (1);
        \coordinate [label=left: {$a_k$}] () at (5);
        \coordinate [label=right: {$b_1$}] () at (10);
        \coordinate [label=right: {$b_2$}] () at (11);
        \coordinate [label=right: {$b_k$}] () at (15);
        \coordinate [label=above: {$c_1$}] () at (6);
        \coordinate [label=above: {$c_2$}] () at (7);
        \coordinate [label={[label distance=1.2]below: {$c_{k-1}$}}] () at (9);
        \coordinate [label=above left: {$x$}] () at (16);
        \coordinate [label=above right: {$y$}] () at (17);

    \end{tikzpicture}
    \caption{$v\in\{c_2,c_4,\ldots, c_{k-2}\}$}\label{subfig:f}
    \end{subfigure}
    \caption{Graph $G_k$ and HISTs of vertex-deleted subgraphs of $G_k$.} 
\label{fig:1}
\end{figure}
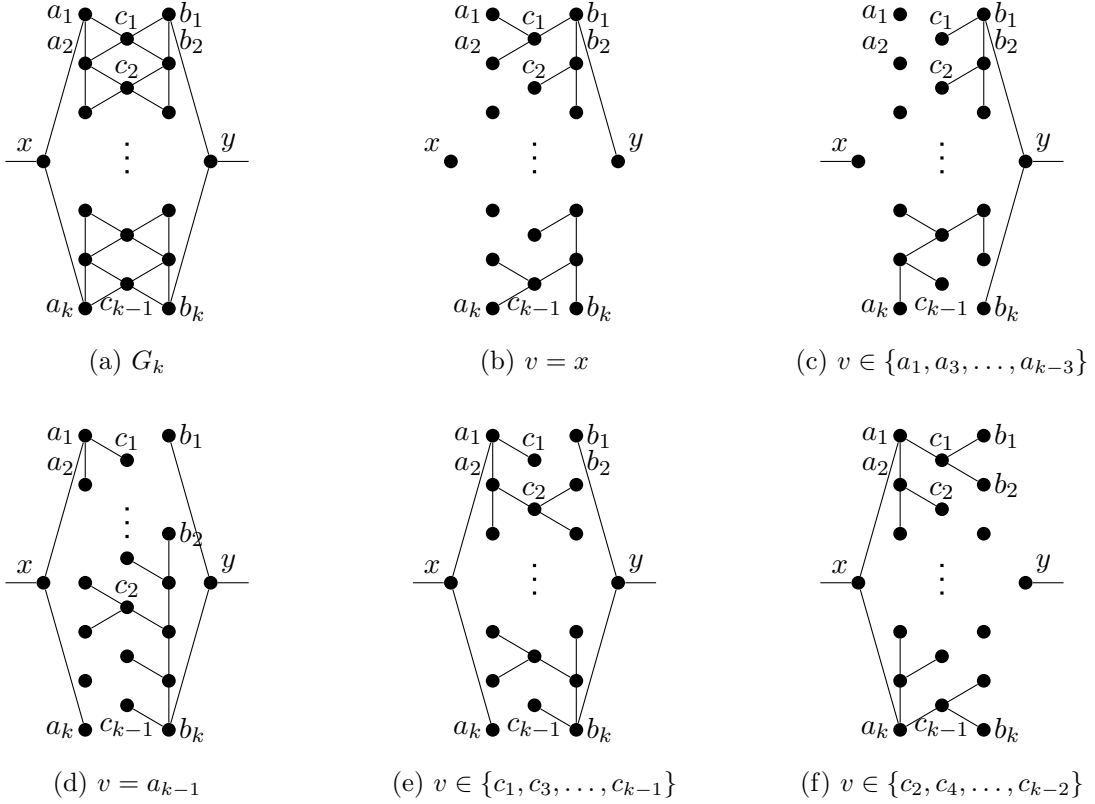

\subsection{A near characterisation of the orders for which HIST-critical graphs exist}
\label{sec:charcterisation}

We summarise our computations and theoretical arguments regarding the existence and non-existence of HIST-critical graphs in the following result.

\begin{thm}\label{orders} 
\textit{Let ${\cal N} := \{ 1, 2, 4, 5, 6, 8, 10, 12\}$ and ${\cal M} := \{26, 30, 34, 38, 45, 48, 52\}$. There exist HIST-critical graphs for every $n \in \mathbb{N} \setminus ({\cal N} \cup {\cal M})$, while there are no HIST-critical graphs of order $n \in {\cal N}$. There exist planar HIST-critical graphs of order $3$, $7$, $11$, $15$, $17$ and $3k + 1$ for every even integer $k \ge 4$, while there are no such graphs of order $n \in {\cal N} \cup \{ 14 \}$.} 
\end{thm}

\begin{proof}

We first prove the statements regarding the general (i.e.\ not necessarily planar) case. The exhaustive computations for HIST-critical graphs whose results were tabulated in Table~\ref{tab:2-conn_HIST-critical} give the non-existence of HIST-critical graphs of order $n$ for every $n \in {\cal N}$. By Theorem~\ref{fragment-chain} and Proposition~\ref{fragments}, for any non-negative integers $k_1$ and $k_2$ with $k_1 + k_2 \ge 2$ we obtain HIST-critical graphs of order $8k_1 + 12k_2 - (k_1 + k_2) = 7k_1 + 11k_2$. From the proof of Theorem~\ref{thm1} we obtain HIST-critical graphs of order $3k+1$ for every even integer $k \ge 4$. It is elementary to verify that we thus obtain the theorem's first statement, using Table~\ref{tab:2-conn_HIST-critical} and the subsequent remark on HIST-critical graphs of girth $8$ and $9$. 

In a very similar way, the statement regarding planar HIST-critical graphs follows from Theorem~\ref{thm1}, Table~\ref{tab:counts_hist-crit_2conn_planar} and by verifying for Table~\ref{tab:2-conn_HIST-critical} which graphs are also planar. This yields extra examples of order $15$ and $17$. 

\end{proof}

It remains an open question whether there exists a HIST-critical graph of order $n$ for $n \in \{26, 30, 34, 38, 45, 48, 52\}$.

\section{On a conjecture of Malkevitch}
\label{sec:malkevitch}

Malkevitch conjectured in 1979~\cite{Ma79} that every planar $4$-connected graph has a HIST. We computationally determined the following.
\begin{prop}
    \textit{Every planar $4$-connected graph up to and including order $22$ has a HIST.}
\end{prop}

Using \texttt{plantri}~\cite{BM07} we generated all planar $4$-connected graphs up to order $22$ and determined none of these were HIST-free using our algorithm in Section~\ref{sec:algo}. The number of planar $4$-connected graphs for each order can be found in Table~\ref{tab:4-conn_planar} in Appendix~\ref{app:4-conn_planar}.
These counts extend the corresponding entry in the Online Encyclopedia of Integer Sequences which were previously only known up to 17 vertices (see: \url{https://oeis.org/A007027}).


It is natural to ask, if all of these graphs contain a HIST, how many HISTs such a graph should necessarily have. Denote by $p(n)$ the minimum number of HISTs in a planar $4$-connected graph of order $n$. We summarise these counts in Table~\ref{tab:min_HISTs_planar_4-conn}. For every entry there is always precisely one graph attaining the given number of HISTs.

For the even orders up to order $18$ this minimum is attained by the antiprism (recall that antiprisms only exist for even orders) which motivates the following section, i.e.\ Section~\ref{sect:antiprisms}, where we establish the number of HISTs in an antiprism. 
A drawing of the graphs on odd orders attaining the minimum number of HISTs can be found in Figure~\ref{fig:4-conn_planar_attaining_min_HISTs} in the Appendix.

\begin{table}[!htb]
\centering
    \begin{tabular}{c| c c c c c c c c c c c c c}
        $n$ & $6$ & $7$ & $8$ & $9$ & $10$ & $11$ & $12$ & $13$ & $14$ & $15$ & $16$ & $17$ & $18$\\\hline
        $p(n)$ & $24$ & $30$ & $48$ & $62$ & $80$ & $64$ & $120$ & $156$ & $168$ & $120$ & $224$ & $398$ & $288$
    
    \end{tabular}
    \caption{Minimum number of HISTs in a planar $4$-connected graph of order $n$.}
    \label{tab:min_HISTs_planar_4-conn}
\end{table}


\subsection{Counting HISTs in antiprisms}
\label{sect:antiprisms}
An \textit{antiprism} is a planar 4-connected even-order graph $(V_k,E_k)$ with
\begin{eqnarray*}
	V_k &=& \{v_0, \ldots, v_{k-1}, w_0,\ldots, w_{k-1}\} \\
	E_k &=& \{v_0v_1, \ldots, v_{k-1}v_0, w_0w_1, \ldots, w_{k-1}w_0, v_0w_0, v_0w_1, v_1w_1, v_1w_2, \ldots, v_{k-1}w_{k-1}, v_{k-1}w_0\}.
\end{eqnarray*}
For instance, the antiprism of order 6 is the octahedron. 

\begin{prop}
    \textit{The antiprism of order $2k$ with $k\geq 3$ has exactly $2k(2k-2)$ HISTs.}
\end{prop}
\begin{proof}
    We denote the vertex $V_k$ and edge set $E_k$ of the antiprism $G$ of order $2k$ as above. Henceforth, we will assume that all indices are taken mod.~$k$. We first show that a HIST in an antiprism cannot have a $4$-vertex. Let $T$ be a HIST of $G$ containing a $4$-vertex $v$. 
Then $v$ must have one or two degree $3$ neighbours. Since $T$ is connected, a $3$-vertex cannot have a degree $4$ neighbour other than $v$ in $T$ and there is exactly one $4$-vertex in $T$. Counting the degree sums, we see that the sum is a multiple of four if there is a $4$-vertex, which implies that there is an even number of edges. However, since $T$ is spanning its number of edges should be $2k-1$ which is a contradiction.

    Similarly, using the degree sum, we see that a HIST $T$ must have $k+1$ $1$-vertices and $k-1$ $3$-vertices. It is easy to see that the $3$-vertices induce a subtree $S$ of $T$, otherwise, $T$ would not be connected. Moreover, $S$ is an induced subgraph of $G$. Indeed, let $u$ and $v$ be two non-adjacent $3$-vertices of $T$ which share an edge in the antiprism. Let $u = v_i$ and $v = v_{i+1}$, then all incident edges of $u$ and $v$ except for $uv$ are in $T$. Then since $w_{i+1}$ has at least two incident edges, it is of degree $3$ and either $w_{i}w_{i+1}$ or $w_{i+1}w_{i+2}$ should be in $T$, but either of these lead to a cycle in $T$. It is straightforward to check that the other cases also lead to a cycle in $T$. It follows that $S$ is an induced path of order $k-1$. 

    Let $S$ be an induced path with endpoint $w_i$ and either $w_iv_i$ or $w_iw_{i+1}$ 
    lie in $S$. In order for $v_{i-2}$ to be in $T$, it needs to be a $1$-vertex, since such a path $S$ can only include $v_{i-2}$ if its order is $k$. Hence, either $v_{i-3}$ or $w_{i-2}$ lie in $S$. For the former case, we have $k-2$ possibilities. These paths are of the form $w_iw_{i+1}\cdots w_jv_jv_{j+1}\cdots v_{i-3}$, where $j = {i, i+1, \ldots, i-3}$. For the latter case, there is only one option, the path $w_iw_{i+1}\cdots w_{i-2}$.

    Assume $S$ to be of the form $w_iw_{i+1}\cdots w_jv_jv_{j+1}\cdots v_{i-3}$. In order for $T$ to be a HIST, $w_j$ must either have a neighbour $w_{j+1}$ or $v_{j-1}$ in $T$. Both of these options completely fix $T$, hence there are precisely two HISTs for every such path $S$.

    Let $S$ be of the form $w_iw_{i+1}\cdots w_{i-2}$. In order for $T$ to be a HIST, $w_i$ must either have neighbours $w_{i-1}$ and $v_{i-1}$ in $T$ or $v_{i-1}$ and $v_{i}$. Both of these options completely fix $T$, hence we also have two HISTs in this case.

    Letting the chosen endpoint be any of the $2k$ vertices of our antiprism, we get 
    \[2k(2(k-2) + 2) = 2k(2k-2)\]
    HISTs. Finally, note that the paths $S$ with endpoint $w_i$ and edges $w_iw_{i-1}$ or $w_iv_{i-1}$ are also paths of the form above but with a different endpoint. Hence, we can conclude that we counted all HISTs of the antiprism. \end{proof}

\subsection{4-regular graphs with or without HISTs}
\label{sec:4-reg-HIST-free}

The fact that the triangle is HIST-free leads to the question whether other $2r$-regular HIST-free graphs exist; this problem was first formulated by Albertson, Berman, Hutchinson, and Thomassen~\cite{ABHT90}. They give an infinite family of 4-regular HIST-free graphs. Such a family of planar graphs had been given earlier by Joffe~\cite{Jo82} but the relevant part of Joffe's thesis cannot be accessed, so hereunder we give a proof of this result. 
We remark that the aforementioned question remains open for $r > 2$. In~\cite{ABHT90}, they were particularly interested in the answer to the question whether connected 6-regular HIST-free graphs exist. They provide infinite families of $(2r+1)$-regular HIST-free graphs for every natural number $r$.

Concerning HIST-critical graphs, we know that exactly one 2-regular such graph exists, and that no such graphs exist that are 3-regular. Whether $k$-regular HIST-critical graphs for $k > 3$ exist is unknown. Theorem~\ref{thm1} shows that there are infinitely many planar HIST-critical graphs with all but four vertices quartic. Unfortunately, we could not find a 4-regular HIST-critical graph, planar or not. 

On the one hand, Malkevitch conjectured that every planar 4-connected graph has a HIST~\cite{Ma79}. On the other hand, he remarked in \cite[Remark 3]{Ma79}, without giving a proof, that there exist planar 3-connected 4-regular graphs that are HIST-free. We now describe such graphs; in particular, we consider the line graph of cubic graphs. 


\begin{prop}
\label{3-conn_counterexample}
\textit{There exist infinitely many $3$-connected $4$-regular HIST-free planar graphs that are the line graph of cubic graphs.}
\end{prop}


\begin{lem}
\label{correspondence}
\textit{Let $G$ be a cubic graph of order $4k+2$. Then there is a $1$-to-$3k$ correspondence between an induced tree $T$ in $G$ such that, for every edge in $E(G) \setminus E(T)$, precisely one of its ends lies on $T$, and the HISTs of the line graph $L(G)$. In particular, the set $E(T)$ coincides with the set of $3$- or $4$-vertices in a HIST of $L(G)$.}
\end{lem}

\begin{proof}
The line graph $L(G)$ has $6k+3$ vertices and is $4$-regular where every edge is in precisely one triangle that surrounds a vertex in $G$. Let $T'$ be a HIST of $L(G)$. By the same argument as in the HIST-freeness proof of Theorem~\ref{thm1}, $T'$ should have a 4-vertex and the other vertices are of degree~1 or 3. Let $S$ be the edges of $G$ corresponding to the 3- or 4-vertices in $T'$. 
The size of $S$ can be calculated as follows: by the degree sum formula $4\times 1 + 3\times (|S|-1) +1\times (6k+3-|S|) = 2|E(T')| = 2(6k+2)$, and hence $|S| = 3k$. 
Since $T'$ is connected, $S$ induces a connected graph $G[S]$. On the one hand, $G[S]$
is adjacent to at most $3k+3$ edges in $E(G) \setminus S$ since $G$ is cubic.
On the other hand, every edge in $E(G) \setminus S$, since it corresponds to a 1-vertex of $T'$, should be adjacent to an edge of $S$ since any 1-vertex is adjacent to a 3- or 4- vertex in $T'$. Hence, 
$|E(G) \setminus S| = 6k+3-3k=3k+3$ implies that, for every $e \in E(G) \setminus S$, precisely one end of $e$ lies on $G[S]$, and therefore $S$ induces a tree in $G$ with the desired property.

Next, let $T$ be an induced tree of $G$ such that for every edge in $E(G) \setminus E(T)$ precisely one of its ends lies on $T$. 
Since any edge in $E(G) \setminus E(T)$ is incident to precisely one vertex in $V(G) \setminus V(T)$, we have $|E(G)| = (|V(T)|-1) + 3(4k+2-|V(T)|) = 6k+3$ and, hence, $|V(T)|=3k+1$. For any edge $e$ in $T$, one can take a tree $T'$ as a subgraph of $L(G)$ recursively: first $T'$ is $K_{1,4}$ with the 4-vertex corresponding to $e$, and next add two pendant edges to a 1-vertex $v$ in $T'$ if $v$ corresponds to an edge in $T$ (note that the choice of two edges is unique since the adding of the other edge makes a triangle in $T'$). This recursive construction finally makes $T'$ a HIST of $L(G)$. Note that $T'$ does not depend on the order of $1$-vertices $v$. Since $|E(T)|=3k$, there are $3k$ choices of $e$ and there is a $1$-to-$3k$ correspondence between $T$ and HISTs of $L(G)$.
\end{proof}

\begin{figure}[th]
	\centering
	\includegraphics[width=12cm]{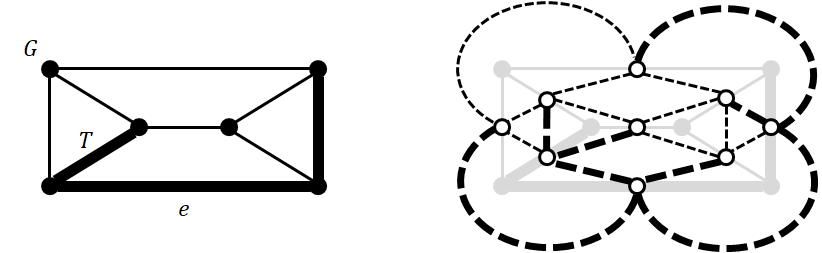}
	\caption{Left: induced tree $T$ (bold edges) in a cubic graph $G$ on which, for every edge in $E(G) \setminus E(T)$, precisely one of its ends lies.  Right: HIST (dotted bold edges) in the line graph $L(G)$ with a 4-vertex corresponding to $e$.}
	\label{fig:3}
\end{figure}	

For an illustration of Lemma~\ref{correspondence}, see Figure~\ref{fig:3}. Note that $T$ gives a partition $V(T) \cup (V(G) \setminus V(T))$ of $V(G)$ such that $V(T)$ induces a tree of order $3k+1$ and $V(G) \setminus V(T)$ is an independent set of order $k+1$.

\begin{proof}[Proof of Proposition~\ref{3-conn_counterexample}]
Let $G$ be a 3-edge connected cubic planar graph of order $4\ell+2~(\ell \ge 1)$ and $H_G$ be its truncated graph, which is obtained from $G$ by replacing each vertex with a triangle. Then $H_G$ is a 3-edge connected cubic planar graph of order $12\ell +6 = 4(3\ell +1)+2$. Now $H_G$ does not have an induced tree of order $8\ell +5$ since every set of $8\ell +5 = 2(4\ell +2)+1$ vertices should have three vertices of a triangle. By Lemma~\ref{correspondence}, we see that the line graph $L(H_G)$ does not have a HIST since there does not exist an induced tree of $H_G$ on $3(3\ell +1)+1 = 9\ell +4 \ge 8\ell +5$ vertices. Thus, $L(H_G)$ is a 3-connected 4-regular HIST-free planar graph. 
\end{proof}

For the smallest example in the proof of Proposition~\ref{3-conn_counterexample}, let $G$ be the graph depicted on the left-hand side of Figure~\ref{fig:3}. Then, $\ell=1$, and we get a 3-connected 4-regular HIST-free planar graph $L(H_G)$ of order $|E(H_G)|=27$. 

\bigskip
In Lemma~\ref{correspondence}, if $|V(G)| = 4k$, then the following analogue holds. Here, a \emph{unicyclic graph} $H$ is a graph with precisely one cycle (say, $C_H$). It can be proved by an argument similar to the proof of Lemma \ref{correspondence}. Therefore, we leave the proof to the reader.

\begin{prop}
\label{analogue}
\textit{Let $G$ be a cubic graph of order $4k$. Then there is a $1$-to-$2|E(C_U)|$ correspondence between a connected induced unicyclic subgraph $U$ of $G$ such that, for every edge in $E(G) \setminus E(U)$, precisely one of its ends lies on $U$, and the HISTs of $L(G)$. 
In particular, for every edge $e \in E(C_U)$, there are precisely two HISTs of $L(G)$ such that the set $E(C_U) \setminus \{ e \}$ coincides with the set of vertices of degree $3$ in the HISTs.
%
}
\end{prop}

Note that $U$ gives a partition $V(U) \cup (V(G) \setminus V(U))$ of $V(G)$ such that $V(U)$ induces a unicyclic graph of order $3k$ and $V(G) \setminus V(U)$ is an independent set of order $k$.

\bigskip

For the 4-connected case, the above method cannot be used to find a 4-regular HIST-free planar graph (i.e., a counterexample to Malkevitch's conjecture), due to the following argument and theorems. In other words, we give a partial affirmative answer to Malkevitch's conjecture for the family of the line graphs of cyclically 4-edge connected cubic graphs. Note that the argument below holds regardless of planarity.

\begin{thm}[Jaeger, \cite{Ja74}]\label{Jaeger}
\textit{Let $G$ be a connected cubic graph of order $n$ and $s(G)$ the maximum number of vertices in a vertex-induced forest of $G$. Then
	$$s(G) \le \left\lfloor \frac{3n-2}{4} \right\rfloor. \qquad (\dagger)$$ 
}
\end{thm}

\begin{thm}[Payan and Sakarovitch, \cite{PS75}]\label{PandS}
\textit{If $G$ is a cyclically $4$-edge-connected cubic graph, then equality holds in $(\dagger)$ in Theorem~\ref{Jaeger}.}
\end{thm}

For a cubic graph $G$, it is well-known and easy to see that the line graph $L(G)$ is 4-connected if and only if $G$ is cyclically 4-edge connected.
Let $G$ be a cyclically 4-edge connected cubic graph of order $n=4k+2$ (resp. $4k$). 
Then by Theorems \ref{Jaeger} and \ref{PandS}, $G$ has an induced forest $F$ with $\left\lfloor \frac{3n-2}{4} \right\rfloor = 3k+1$ (resp.~$3k-1$) vertices. 
When $n=4k+2$, we have $|E(F)| \le 3k$ and $|E(G) \setminus E(F)| \le 3(k+1)$, both of which should attain the equality since $|E(G)| = 6k+3$. 
This implies that $F$ is a tree such that, for every edge in $E(G) \setminus E(F)$, precisely one of its ends lies on $F$. 
When $n=4k$, we have $|E(F)| \le 3k-2$ and $|E(G) \setminus E(F)| \le 3(k+1)$, which implies that one of the following cases occurs: 

(a) $F$ is a tree and $V(G) \setminus V(F)$ induces precisely one edge, say $uv$, or
(b) $F$ is a forest with precisely two components and $V(G) \setminus V(F)$ is an independent set. 
For (a), let $F'$ be the graph induced by $V(F) \cup \{u\}$. 
For (b), let $F'$ be a graph induced by $V(F) \cup \{w\}$, where $w$ has neighbours in each of two components of $F$. 
(Such $w$ exists since $G$ is connected.) 
In both cases, $F'$ is an induced unicyclic subgraph such that, for every edge in $E(G) \setminus E(F)$, precisely one of its ends lies on $F$. 
Thus, in both cases, it follows from Lemma~\ref{correspondence} and Proposition~\ref{analogue} that the following proposition holds.

\begin{prop}
\textit{For every cyclically $4$-edge connected cubic graph, its ($4$-connected $4$-regular) line graph has a HIST.} 
\end{prop}


\subsection{4-connected HIST-free graphs of small genus}
We could not find a counterexample to Malkevitch's conjecture, so we tried to describe a $4$-connected HIST-free graph of small genus. 
Here the (orientable) genus of a graph $G$ is the smallest integer $g \ge 0$ such that $G$ can be embedded on the orientable surface $\mathbb{S}_g$ of genus $g$. Let $G$ be a 4-regular graph. 
The \emph{$K_4$-inflation} of $G$ is to replace each vertex $v$ of $G$ with $K_4$, 
and to join suitable two vertices of two $K_4$'s so that the new edges are in 1-to-1 correspondence with the edges in $G$, see Figure~\ref{fig:replacing}. 

\setcounter{figure}{2}
\begin{figure}[!htb]
\centering
    \begin{subfigure}[t]{\linewidth}
    \centering
  \begin{tikzpicture}
        \node[vertex] (0) at (0, 0) {};
        \node[vertex] (1) at (6.5, 0) {};
        \node[vertex] (2) at (7.5, 0) {};
        \node[vertex] (3) at (7, 0.5) {};
        \node[vertex] (4) at (7, -0.5) {};

        \path[draw]
        (-1, 0) edge node {} (1, 0) 
        (0, -1) edge node {} (0, 1) 
        (6, 0) edge node {} (8, 0) 
        (7, -1) edge node {} (7, 1) 
        (1) edge node {} (3) 
        (3) edge node {} (2) 
        (2) edge node {} (4) 
        (4) edge node {} (1) 
        ;

\draw[very thick,->] (3,0)--(4,0);

\coordinate [label=above left: {$v$}] () at (0);
  \end{tikzpicture}
  \end{subfigure}
    \caption{Replacing a vertex $v$ with $K_4$.}
\label{fig:replacing}
\end{figure}

\begin{figure}[ht]
\centering
\begin{tikzpicture}[scale=0.7]
\def\x{2*sqrt(3)}
\def\y{-3}
\foreach \k in{1,...,3}
 {
    \coordinate (A_\k) at ({-sin(\k*120)},{cos(\k*120)});
    \coordinate (B_\k) at ({-2*sin(\k*120)},{2*cos(\k*120)});
    \coordinate (C_\k) at ({-3*sin(\k*120)},{3*cos(\k*120)});
    \coordinate (D_\k) at ({\x*cos(\k*120)-\y*sin(\k*120)},{\x*sin(\k*120)+\y*cos(\k*120)});
    \coordinate (E_\k) at ({-\y*sin(\k*120)},{\y*cos(\k*120)});
    \coordinate (F_\k) at ({-\x*cos(\k*120)-\y*sin(\k*120)},{-\x*sin(\k*120)+\y*cos(\k*120)});
   \draw[fill=black] (A_\k) circle (2pt);
   \draw[fill=black] (B_\k) circle (2pt);
   \draw[fill=black] (C_\k) circle (2pt);
   \draw[fill=black] (D_\k) circle (2pt);
   \draw[fill=black] (E_\k) circle (2pt);
   \draw[fill=black] (F_\k) circle (2pt);
 };

\draw[thick] (A_1)--(A_2)--(A_3)--(A_1);
\draw[thick] (C_1)--(D_2)--(F_3)--(C_1);
\draw[thick] (C_2)--(D_3)--(F_1)--(C_2);
\draw[thick] (C_3)--(D_1)--(F_2)--(C_3);
\draw[thick] (A_1) to [bend left] (B_1) to [bend left] (A_1);
\draw[thick] (A_2) to [bend left] (B_2) to [bend left] (A_2);
\draw[thick] (A_3) to [bend left] (B_3) to [bend left] (A_3);
\draw[thick] (B_1) to [bend left] (C_1) to [bend left] (B_1);
\draw[thick] (B_2) to [bend left] (C_2) to [bend left] (B_2);
\draw[thick] (B_3) to [bend left] (C_3) to [bend left] (B_3);
\draw[thick] (D_1) to [bend left] (E_1) to [bend left] (D_1);
\draw[thick] (D_2) to [bend left] (E_2) to [bend left] (D_2);
\draw[thick] (D_3) to [bend left] (E_3) to [bend left] (D_3);
\draw[thick] (E_1) to [bend left] (F_1) to [bend left] (E_1);
\draw[thick] (E_2) to [bend left] (F_2) to [bend left] (E_2);
\draw[thick] (E_3) to [bend left] (F_3) to [bend left] (E_3);
\end{tikzpicture}

\caption{A 4-edge-connected 4-regular planar multigraph $G$ without a hamiltonian path.}
\label{fig:nonhamiltonian}
\end{figure}

\begin{figure}[ht]
\centering
\begin{minipage}[h]{0.49\columnwidth}
\centering
\begin{tikzpicture}[scale=0.45]
\def\p{0.5}
\def\q{1}
\def\s{-6}
\foreach \k in{0,...,2}
  \foreach \l in{0,...,5}
 {
    \coordinate (A_\k\l) at ({\p*cos(\k*120)-(\q+\l)*sin(\k*120)},{\p*sin(\k*120)+(\q+\l)*cos(\k*120)});
    \coordinate (B_\k\l) at ({-\p*cos(\k*120)-(\q+\l)*sin(\k*120)},{-\p*sin(\k*120)+(\q+\l)*cos(\k*120)});
    \draw[fill=black] (A_\k\l) circle (2pt);
    \draw[fill=black] (B_\k\l) circle (2pt);
    \draw[thick] (A_\k\l)--(B_\k\l);
 };

\foreach \k in{0,...,2}
  \foreach \r in{-5,-4,-3,-2,4,5}
 {
    \coordinate (C_\k\r) at ({\r*cos(\k*120)-\s*sin(\k*120)},{\r*sin(\k*120)+\s*cos(\k*120)});
    \coordinate (D_\k\r) at ({\r*cos(\k*120)-(\s+1)*sin(\k*120)},{\r*sin(\k*120)+(\s+1)*cos(\k*120)});
    \draw[fill=black] (C_\k\r) circle (2pt);
    \draw[fill=black] (D_\k\r) circle (2pt);
    \draw[thick] (C_\k\r)--(D_\k\r);
 };

\draw[thick] (A_00)--(B_20);
\draw[thick] (A_10)--(B_00);
\draw[thick] (A_20)--(B_10);
\draw[thick] (A_05)--(D_15);
\draw[thick] (A_15)--(D_25);
\draw[thick] (A_25)--(D_05);
\draw[thick] (B_05)--(D_2-5);
\draw[thick] (B_15)--(D_0-5);
\draw[thick] (B_25)--(D_1-5);
\draw[thick] (C_0-5)--(C_25);
\draw[thick] (C_1-5)--(C_05);
\draw[thick] (C_2-5)--(C_15);

\foreach \k in{0,...,2}
 {
  \draw[thick] (A_\k0)--(A_\k5);
  \draw[thick] (B_\k0)--(B_\k5);
  \draw[thick] (A_\k0)--(B_\k1);
  \draw[thick] (A_\k2)--(B_\k3);
  \draw[thick] (A_\k4)--(B_\k5);
  \draw[thick] (B_\k0)--(A_\k1);
  \draw[thick] (B_\k2)--(A_\k3);
  \draw[thick] (B_\k4)--(A_\k5);
  \draw[thick] (C_\k-5)--(C_\k5);
  \draw[thick] (D_\k-5)--(D_\k5);
  \draw[thick] (C_\k-5)--(D_\k-4);
  \draw[thick] (C_\k-3)--(D_\k-2);
  \draw[thick] (C_\k4)--(D_\k5);
  \draw[thick] (D_\k-5)--(C_\k-4);
  \draw[thick] (D_\k-3)--(C_\k-2);
  \draw[thick] (D_\k4)--(C_\k5);
 };
\end{tikzpicture}
\end{minipage}
\begin{minipage}[h]{0.49\columnwidth}
%
\centering
\begin{tikzpicture}[scale=0.45]
\def\p{0.5}
\def\q{1}
\def\s{-6}
\foreach \k in{0,...,2}
  \foreach \l in{0,...,5}
 {
    \coordinate (A_\k\l) at ({\p*cos(\k*120)-(\q+\l)*sin(\k*120)},{\p*sin(\k*120)+(\q+\l)*cos(\k*120)});
    \coordinate (B_\k\l) at ({-\p*cos(\k*120)-(\q+\l)*sin(\k*120)},{-\p*sin(\k*120)+(\q+\l)*cos(\k*120)});
    \draw[fill=black] (A_\k\l) circle (2pt);
    \draw[fill=black] (B_\k\l) circle (2pt);
    \draw[thick] (A_\k\l)--(B_\k\l);
 };

\foreach \k in{0,...,2}
  \foreach \r in{-5,-4,-3,-2,4,5}
 {
    \coordinate (C_\k\r) at ({\r*cos(\k*120)-\s*sin(\k*120)},{\r*sin(\k*120)+\s*cos(\k*120)});
    \coordinate (D_\k\r) at ({\r*cos(\k*120)-(\s+1)*sin(\k*120)},{\r*sin(\k*120)+(\s+1)*cos(\k*120)});
    \draw[fill=black] (C_\k\r) circle (2pt);
    \draw[fill=black] (D_\k\r) circle (2pt);
 };

\draw[thick] (A_00)--(B_20);
\draw[thick] (A_10)--(B_00);
\draw[thick] (A_20)--(B_10);
\draw[thick] (A_05)--(D_15);
\draw[thick] (A_15)--(D_25);
\draw[thick] (A_25)--(D_05);
\draw[thick] (B_05)--(D_2-5);
\draw[thick] (B_15)--(D_0-5);
\draw[thick] (B_25)--(D_1-5);
\draw[thick] (C_0-5)--(C_25);
\draw[thick] (C_1-5)--(C_05);
\draw[thick] (C_2-5)--(C_15);

\foreach \k in{0,...,2}
 {
  \draw[thick] (A_\k0)--(A_\k5);
  \draw[thick] (B_\k0)--(B_\k5);
  \draw[thick] (A_\k0)--(B_\k1);
  \draw[thick] (A_\k2)--(B_\k3);
  \draw[thick] (A_\k4)--(B_\k5);
  \draw[thick] (C_\k-5)--(D_\k-4);
  \draw[thick] (C_\k-3)--(D_\k-2);
  \draw[thick] (C_\k4)--(D_\k5);
 };

\def\t{1.5} 
\def\u{-6.5} 
\def\v{3}
\foreach \k in{0,...,2}
 {
  \coordinate (a_\k) at ({-\t*cos(\k*120)-\u*sin(\k*120)},{-\t*sin(\k*120)+\u*cos(\k*120)});
  \coordinate (b_\k) at ({\t*cos(\k*120)-\u*sin(\k*120)},{\t*sin(\k*120)+\u*cos(\k*120)});
  \coordinate (c_\k) at ({\t*cos(\k*120)-(\u+2)*sin(\k*120)},{\t*sin(\k*120)+(\u+2)*cos(\k*120)});
  \coordinate (d_\k) at ({-\t*cos(\k*120)-(\u+2)*sin(\k*120)},{-\t*sin(\k*120)+(\u+2)*cos(\k*120)});
  \coordinate (a'_\k) at ({-\t*cos(\k*120)-(\s*sin(\k*120)},{-\t*sin(\k*120)+\s*cos(\k*120)});
  \coordinate (b'_\k) at ({\t*cos(\k*120)-\s*sin(\k*120)},{\t*sin(\k*120)+\s*cos(\k*120)});
  \coordinate (c'_\k) at ({\t*cos(\k*120)-(\s+1)*sin(\k*120)},{\t*sin(\k*120)+(\s+1)*cos(\k*120)});
  \coordinate (d'_\k) at ({-\t*cos(\k*120)-(\s+1)*sin(\k*120)},{-\t*sin(\k*120)+(\s+1)*cos(\k*120)});
  \coordinate (a^_\k) at ({-\u*sin(\k*120)},{\u*cos(\k*120)});
  \coordinate (b^_\k) at ({\t*cos(\k*120)-(\u+1)*sin(\k*120)},{\t*sin(\k*120)+(\u+1)*cos(\k*120)});
  \coordinate (c^_\k) at ({-(\u+2)*sin(\k*120)},{(\u+2)*cos(\k*120)});
  \coordinate (d^_\k) at ({-\t*cos(\k*120)-(\u+1)*sin(\k*120)},{-\t*sin(\k*120)+(\u+1)*cos(\k*120)});
  \coordinate (e'_\k) at ({\v*cos(\k*120)-8*sin(\k*120)},{\v*sin(\k*120)+8*cos(\k*120)});
  \coordinate (f'_\k) at ({\v*cos(\k*120)-8.2*sin(\k*120)},{\v*sin(\k*120)+8.2*cos(\k*120)});
  \coordinate (g'_\k) at ({\v*cos(\k*120)-8.4*sin(\k*120)},{\v*sin(\k*120)+8.4*cos(\k*120)});
  \coordinate (e^_\k) at ({-\v*cos(\k*120)-8*sin(\k*120)},{-\v*sin(\k*120)+8*cos(\k*120)});
  \coordinate (f^_\k) at ({-\v*cos(\k*120)-8.2*sin(\k*120)},{-\v*sin(\k*120)+8.2*cos(\k*120)});
  \coordinate (g^_\k) at ({-\v*cos(\k*120)-8.4*sin(\k*120)},{-\v*sin(\k*120)+8.4*cos(\k*120)});
  \draw[red,very thick] (a_\k)--(b_\k)--(c_\k)--(d_\k)--(a_\k);
  \draw[red,->] (a_\k)--(a^_\k);
  \draw[red,->] (d_\k)--(c^_\k);
  \draw[red,->>] (a_\k)--(d^_\k);
  \draw[red,->>] (b_\k)--(b^_\k);
  \draw[thick] (C_\k-5)--(a'_\k);
  \draw[thick] (D_\k-5)--(d'_\k);
  \draw[thick] (C_\k5)--(b'_\k);
  \draw[thick] (D_\k5)--(c'_\k);
  \draw[thick] (D_\k-3) to [bend left] ({-\t*5/5*cos(\k*120)-(\u+2)*sin(\k*120)},{-\t*5/5*sin(\k*120)+(\u+2)*cos(\k*120)});
  \draw[thick] (D_\k-5) to [bend left] ({-\t*4/5*cos(\k*120)-(\u+2)*sin(\k*120)},{-\t*4/5*sin(\k*120)+(\u+2)*cos(\k*120)});
  \draw[thick] (B_\k4) to [bend left] ({-\t*3/5*cos((\k+2)*120)-(\u+2)*sin((\k+2)*120)},{-\t*3/5*sin((\k+2)*120)+(\u+2)*cos((\k+2)*120)});
  \draw[thick] (B_\k2) to [bend left] ({-\t*2/5*cos((\k+2)*120)-(\u+2)*sin((\k+2)*120)},{-\t*2/5*sin((\k+2)*120)+(\u+2)*cos((\k+2)*120)});
  \draw[thick] (B_\k0) to [bend left] ({-\t*1/5*cos((\k+2)*120)-(\u+2)*sin((\k+2)*120)},{-\t*1/5*sin((\k+2)*120)+(\u+2)*cos((\k+2)*120)});
  \draw[thick] (A_\k1) to [bend right] ({\t*1/5*cos((\k+1)*120)-(\u+2)*sin((\k+1)*120)},{\t*1/5*sin((\k+1)*120)+(\u+2)*cos((\k+1)*120)});
  \draw[thick] (A_\k3) to [bend right] ({\t*2/5*cos((\k+1)*120)-(\u+2)*sin((\k+1)*120)},{\t*2/5*sin((\k+1)*120)+(\u+2)*cos((\k+1)*120)});
  \draw[thick] (A_\k5) to [bend right] ({\t*3/5*cos((\k+1)*120)-(\u+2)*sin((\k+1)*120)},{\t*3/5*sin((\k+1)*120)+(\u+2)*cos((\k+1)*120)});
  \draw[thick] (D_\k4) to [bend right] ({\t*4/5*cos(\k*120)-(\u+2)*sin(\k*120)},{\t*4/5*sin(\k*120)+(\u+2)*cos(\k*120)});
  \draw[thick] (C_\k-2) to [bend right] ({-\t*5/5*cos(\k*120)-\u*sin(\k*120)},{-\t*5/5*sin(\k*120)+\u*cos(\k*120)});
  \draw[thick] (C_\k-4) to [bend right] ({-\t*4/5*cos(\k*120)-\u*sin(\k*120)},{-\t*4/5*sin(\k*120)+\u*cos(\k*120)});
  \draw[thick] (C_\k5) to [bend left] ({\t*4/5*cos(\k*120)-\u*sin(\k*120)},{\t*4/5*sin(\k*120)+\u*cos(\k*120)});
  \draw[thick] ({\t*3/5*cos((\k+1)*120)-\u*sin((\k+1)*120)},{\t*3/5*sin((\k+1)*120)+\u*cos((\k+1)*120)}) to [bend right] (e'_\k);
  \draw[thick] (e'_\k) to [bend right] (e^_\k);
  \draw[thick] (e^_\k) to [bend right] ({-\t*3/5*cos((\k+2)*120)-\u*sin((\k+2)*120)},{-\t*3/5*sin((\k+2)*120)+\u*cos((\k+2)*120)});
  \draw[thick] ({\t*2/5*cos((\k+1)*120)-\u*sin((\k+1)*120)},{\t*2/5*sin((\k+1)*120)+\u*cos((\k+1)*120)}) to [bend right] (f'_\k);
  \draw[thick] (f'_\k) to [bend right] (f^_\k);
  \draw[thick] (f^_\k) to [bend right] ({-\t*2/5*cos((\k+2)*120)-\u*sin((\k+2)*120)},{-\t*2/5*sin((\k+2)*120)+\u*cos((\k+2)*120)});
  \draw[thick] ({\t*1/5*cos((\k+1)*120)-\u*sin((\k+1)*120)},{\t*1/5*sin((\k+1)*120)+\u*cos((\k+1)*120)}) to [bend right] (g'_\k);
  \draw[thick] (g'_\k) to [bend right] (g^_\k);
  \draw[thick] (g^_\k) to [bend right] ({-\t*1/5*cos((\k+2)*120)-\u*sin((\k+2)*120)},{-\t*1/5*sin((\k+2)*120)+\u*cos((\k+2)*120)});
 };

\foreach \k in{0,...,2}
  \foreach \r in{-5,-4,-3,-2,4,5}
    \draw[thick] (C_\k\r)--(D_\k\r);
\end{tikzpicture}
\end{minipage}
\caption{Left: 4-connected HIST-free graph $H_G$ of genus 3. Right: graph $H_G$ embedded on the orientable surface $\mathbb{S}_3$. The three (red) bold rectangles represent the three handles; in each rectangle, identify two pairs of sides along the arrows, respectively.} 
\label{fig:K_4-inflation}
\end{figure}

\begin{prop}\label{lem:1}
\textit{Let $G$ be a $4$-regular planar multigraph without a hamiltonian path. Then the $K_4$-inflation $H_G$ of $G$ is HIST-free.}
\end{prop}

\begin{proof}
Let $v$ be a vertex of $G$ and $e_1, e_2, e_3, e_4$ be the four edges incident to $v$. 
For a subtree $T$ of $H_G$ without 2-vertices, it is impossible that at least three edges of $\{e_1, e_2, e_3, e_4\}$ are in $T$ if they are connected by the edges of $K_4$ corresponding to $v$. 
So it is not difficult to see that $T$ corresponds to a subpath $P$ of $G$. 
Since $G$ has no hamiltonian path, $P$ cannot span the vertices of $G$ and $T$ cannot span the vertices of $H_G$.
\end{proof}

\begin{thm}
\label{genus3}
\textit{There exist infinitely many $4$-connected HIST-free graphs of genus $3$}.
\end{thm}

\begin{proof}
Let $G$ be the $4$-edge-connected $4$-regular planar multigraph on $18$ vertices, depicted in Figure~\ref{fig:nonhamiltonian}. It is straightforward to see that $G$ has no hamiltonian path. 
By Proposition~\ref{lem:1}, the $K_4$-inflation $H_G$ of $G$ (see the left-hand side of Figure~\ref{fig:K_4-inflation}) is HIST-free. 
Since $G$ is $4$-edge connected, $H_G$ is $4$-connected.
The graph $H_G$ can be embedded on the orientable surface $\mathbb{S}_{3}$ (triple torus) as depicted on the right-hand side of Figure~\ref{fig:K_4-inflation}

%

In~\cite{Br22} Brinkmann developed a practical algorithm for computing the genus of a graph. Using Brinkmann's algorithm we determined that $H_G$ cannot be embedded on an orientable surface of genus at most~2, so the genus of $H_G$ is precisely $3$.

Finally, the graph $H_G$ can easily be extended to an infinite family; in $G$, as depicted in Figure~\ref{fig:nonhamiltonian}, replace two consecutive pairs of parallel edges with any number of pairs of parallel edges.
\end{proof}

\section*{Acknowledgements}
We would like to thank Andreas Awouters for providing us with an independent implementation of the algorithm from Section~\ref{sec:algo} for counting the number of HISTs in a graph.
Jan Goedgebeur and Jarne Renders are supported by Internal Funds of KU Leuven. 
Kenta Noguchi is partially supported by JSPS KAKENHI Grant Number JP21K13831.
Several of the computations for this work were carried out using the supercomputer infrastructure provided by the VSC (Flemish Supercomputer Center), funded by the Research Foundation Flanders (FWO) and the Flemish Government.


\clearpage

\appendix
\section{Appendix}

\subsection{Correctness tests}\label{app:correctness}
We performed various tests for verifying the correctness of the implementation of our algorithm described in Section~\ref{sec:algo} for counting the number of HISTs and testing HIST-criticality. Our implementation is available on GitHub~\cite{GNRZ23}. We will call this Implementation~A. An independent implementation of the algorithm which varies only slightly in the way edges are chosen to be added to the subtree was written by Andreas Awouters. We will call this Implementation~B. 

We also implemented another backtracking algorithm which counts the number of HISTs in a graph based on an algorithm for the generation of spanning trees by Kapoor and Ramesh~\cite{KR95}. We call this Implementation~C. It starts by creating an initial spanning tree $T$ of the graph $G$ using depth first search. This does not need to be a HIST, but to have the correct counts, we check whether it is one. We apply the recursive algorithm to such a tree $T$, where some of the edges of $G$ are marked with ``in'' or with ``out''. ``In'' meaning it will remain in the trees generated by this branch of the search tree. ``Out'' meaning it will never belong to the trees generated by this branch of the search tree.

For a call of the recursive algorithm, we choose a non-edge $e$ of $T$ which has not been marked as ``out'' and compute its \emph{fundamental cycle} in $T$, i.e. the unique cycle containing $e$ in $T+e$. Denote its edges by $e, f_1, \ldots, f_k$. Let $f_i$ be the first edge not marked as ``in''. We mark $e$ as ``in'' and $f_i$ as ``out'' and recursively apply the algorithm to $T+e-f_i$. After this, we mark $f_i$ as ``in'' and mark the next edge $f_j$ not marked ``in'' as ``out'' and apply the algorithm recursively to $T+e-f_j$. Then we also mark $f_j$ as ``in'', etc. until we have done this for all edges not marked ``in'' of the fundamental cycle. Finally, we unmark the edges we just marked as ``in'' and mark $e$ as ``out'' and apply the recursive algorithm to $T$.

Because the edges marked ``in'' and ``out'' lay restrictions on the edges present in the spanning trees we generate in branches of the search tree, we can apply similar pruning rules as described in the algorithm of Section~\ref{sec:algo}. For example when a vertex $v$ has two incident edges marked ``in'' and all other incident edges marked ``out'', then we can already backtrack, since all trees which will be generated in this branch of the search space will have $v$ as a degree $2$ vertex. 

First of all, since all algorithms are backtracking algorithms based on algorithms for generating spanning trees, we can remove the pruning criteria specifically for HISTs and see if they correctly count the number of spanning trees of graphs. We verified this for all algorithms for a large sample of graphs, whose number of spanning trees can easily be computed using Kirchhoff's Matrix Tree Theorem.

We used Implementations~B and C to verify the counts of Table~\ref{tab:2-conn_HIST-critical} obtained by Implementation~A. Since Implementations~B and C are a bit slower than Implementation~A, we were not able to double-check all counts, but we verified the following.

In the general case both B and C verified the counts up to and including order $11$, for girth at least $4$ both implementations verified the counts up to order $14$, for girth at least $5$ both implementations verified the counts up to and including order $17$, for girth at least $6$ both implementations verified the counts up to and including order $19$ and for girth at least $7$ both implementations verified the counts up to and including order $21$.

We also verified the counts of Table~\ref{tab:counts_hist-crit_2conn_planar} obtained by Implementation~A. Both Implementations~B and C verified the counts up to and including order $13$.

Implementations~B and C verified the counts of Table~\ref{tab:min_HISTs_planar_4-conn} up to and including order $14$, hence also verifying the results of Proposition~$4$ up to this order.

\clearpage

\subsection{Smallest HIST-critical graphs}

\begin{figure}[!htb]
    \centering
    \newcommand{\y}{1.5}
    \begin{tikzpicture}
        \node[regular polygon, regular polygon sides=3, minimum size=4cm, draw, black] at (0,0) (triangle) {}; 
        \foreach \i in {1,2,...,3}
            \node[vertex] at (triangle.corner \i) {};
    \end{tikzpicture}\qquad
    \begin{tikzpicture}[scale = 0.5, rotate=90]
        \node[vertex] (0) at (5.142857142857142, -0.14285714285714235) {};
        \node[vertex] (1) at (-0.4593643080279666, -1.9089917715871927) {};
        \node[vertex] (2) at (-0.4593643080279666, 1.6267472985030238) {};
        \node[vertex] (3) at (0.1445926513611342, -5.142857142857142) {};
        \node[vertex] (4) at (0.1445926513611342, 4.857142857142858) {};
        \node[vertex] (5) at (1.841919968264987, -0.14285714285714235) {};
        \node[vertex] (6) at (-4.857142857142858, -0.14285714285714235) {};

         \path[draw]
        (0) edge node {} (3) 
        (0) edge node {} (4) 
        (0) edge node {} (5) 
        (1) edge node {} (3) 
        (1) edge node {} (5) 
        (1) edge node {} (6) 
        (2) edge node {} (4) 
        (2) edge node {} (5) 
        (2) edge node {} (6) 
        (3) edge node {} (6) 
        (4) edge node {} (6) 
        ;

    \end{tikzpicture}\qquad
    \begin{tikzpicture}[scale=0.5]

        \node[vertex] (0) at (5.142857142857142, -0.14115124043476168) {};
        \node[vertex] (1) at (0.14285714285714235, 2.1276989813325535) {};
        \node[vertex] (2) at (-4.857142857142858, -0.14115124043476168) {};
        \node[vertex] (3) at (1.89251607873709, -0.6768046010625337) {};
        \node[vertex] (4) at (0.14285714285714235, -5.142857142857142) {};
        \node[vertex] (5) at (0.14285714285714235, 4.857142857142858) {};
        \node[vertex] (6) at (-1.610212434223348, -0.6768046010625337) {};

         \path[draw]
        (0) edge node {} (3) 
        (0) edge node {} (4) 
        (0) edge node {} (5) 
        (1) edge node {} (3) 
        (1) edge node {} (5) 
        (1) edge node {} (6) 
        (2) edge node {} (4) 
        (2) edge node {} (5) 
        (2) edge node {} (6) 
        (3) edge node {} (4) 
        (3) edge node {} (6) 
        (4) edge node {} (6) 
        ;

    \end{tikzpicture}\qquad
    \begin{tikzpicture}[scale=0.5]

        \node[vertex] (0) at (5.142857142857142, -0.14115124043476168) {};
        \node[vertex] (1) at (0.14285714285714235, 2.1276989813325535) {};
        \node[vertex] (2) at (-4.857142857142858, -0.14115124043476168) {};
        \node[vertex] (3) at (1.89251607873709, -0.6768046010625337) {};
        \node[vertex] (4) at (0.14285714285714235, -5.142857142857142) {};
        \node[vertex] (5) at (0.14285714285714235, 4.857142857142858) {};
        \node[vertex] (6) at (-1.610212434223348, -0.6768046010625337) {};
        \node[vertex] (7) at ($(2)!0.5!(5)$) {};
        \node[vertex] (8) at ($(3)!0.5!(7)$) {};

        \begin{scope}[on background layer]
         \path[draw]
        (0) edge node {} (3) 
        (0) edge node {} (4) 
        (0) edge node {} (5) 
        (1) edge node {} (3) 
        (1) edge node {} (5) 
        (1) edge node {} (6) 
        (2) edge node {} (4) 
        (2) edge node {} (5) 
        (2) edge node {} (6) 
        (3) edge node {} (4) 
        (3) edge node {} (6) 
        (4) edge node {} (6) 
        (3) edge node {} (7)
        ;
        \end{scope}

    \end{tikzpicture}\qquad
    \begin{tikzpicture}[scale=0.5]

        \node[vertex] (0) at (5.142857142857142, -0.14115124043476168) {};
        \node[vertex] (1) at (0.14285714285714235, 2.1276989813325535) {};
        \node[vertex] (2) at (-4.857142857142858, -0.14115124043476168) {};
        \node[vertex] (3) at (1.89251607873709, -0.6768046010625337) {};
        \node[vertex] (4) at (0.14285714285714235, -5.142857142857142) {};
        \node[vertex] (5) at (0.14285714285714235, 4.857142857142858) {};
        \node[vertex] (6) at (-1.610212434223348, -0.6768046010625337) {};
        \node[vertex] (7) at ($(2)!0.5!(5)$) {};
        \node[vertex] (8) at ($(3)!0.5!(7)$) {};

        \begin{scope}[on background layer]
         \path[draw]
        (0) edge node {} (3) 
        (0) edge node {} (4) 
        (0) edge node {} (5) 
        (1) edge node {} (3) 
        (1) edge node {} (5) 
        (1) edge node {} (6) 
        (2) edge node {} (4) 
        (2) edge node {} (5) 
        (2) edge node {} (6) 
        (3) edge node {} (4) 
        (3) edge node {} (6) 
        (4) edge node {} (6) 
        (3) edge node {} (7)
        (0) edge node {} (1)
        ;
        \end{scope}

    \end{tikzpicture}
    \caption{The five HIST-critical graphs with smallest order.}\label{fig:5_smallest_HIST-critical}
\end{figure}
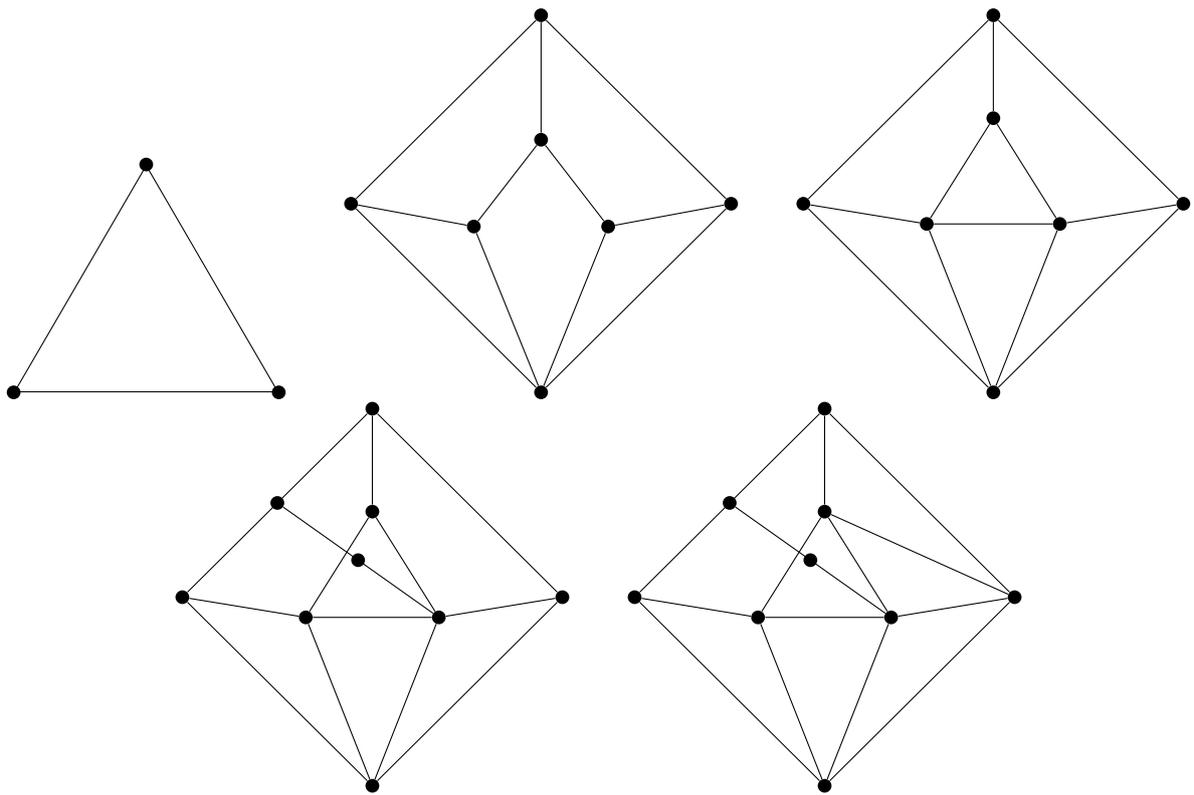

\clearpage
\subsection{HIST-critical graphs with a specific girth}

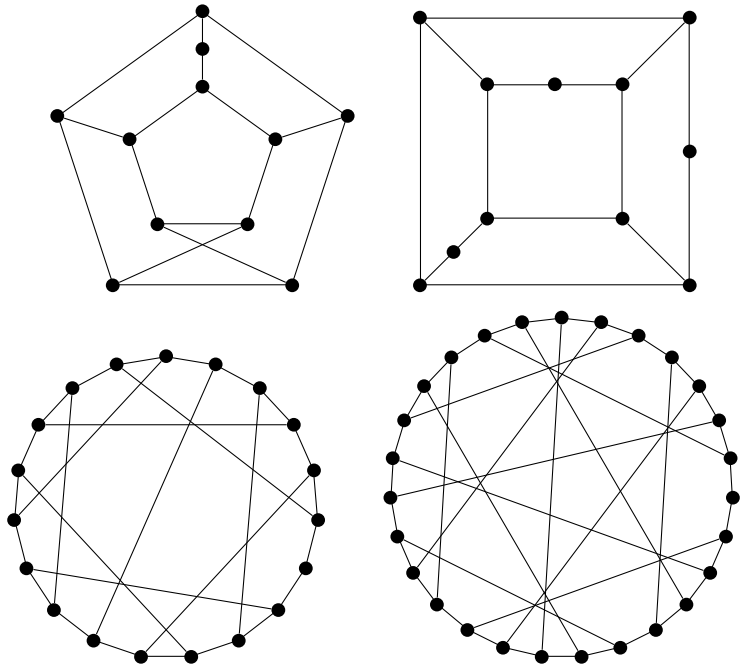
\begin{figure}[!htb]
	\centering
	\begin{tikzpicture}
        \node[regular polygon, regular polygon sides=5, minimum size=2cm, draw, black] at (0,0) (triangle) {}; 
        \foreach \i in {1,2,...,5}
            \node[vertex] (i\i) at (triangle.corner \i) {};
        \node[regular polygon, regular polygon sides=5, minimum size=4cm, draw, black] at (0,0) (t2) {}; 
        \foreach \i in {1,2,...,5}
            \node[vertex] (o\i) at (t2.corner \i) {};

        \draw (i1) to (o1);
        \draw (i5) to (o5);
        \draw (i2) to (o2);
        \draw (i3) to (o4);
        \draw (i4) to (o3);

        \node[vertex] at ($(i1)!0.5!(o1)$) {}; 

	\end{tikzpicture}\qquad
	\begin{tikzpicture}
		\newcommand{\scale}{1.25}	
        \node[regular polygon, regular polygon sides=4, minimum size=\scale*2cm, draw, black] at (0,0) (triangle) {}; 
        \foreach \i in {1,2,...,4}
            \node[vertex] (i\i) at (triangle.corner \i) {};
        \node[regular polygon, regular polygon sides=4, minimum size=\scale*4cm, draw, black] at (0,0) (t2) {}; 
        \foreach \i in {1,2,...,4}
            \node[vertex] (o\i) at (t2.corner \i) {};

        \draw (i1) to (o1);
        \draw (i2) to (o2);
        \draw (i3) to (o3);
        \draw (i4) to (o4);

        \node[vertex] at ($(i1)!0.5!(i2)$) {}; 
        \node[vertex] at ($(i3)!0.5!(o3)$) {}; 
        \node[vertex] at ($(o4)!0.5!(o1)$) {}; 

	\end{tikzpicture}

	\medskip

	\begin{tikzpicture}
		\newcommand{\scale}{4}	
		\node[regular polygon, regular polygon sides=19, minimum size=\scale*1cm, draw, black] at (0,0) (c) {}; 
        \foreach \i in {1,2,...,19}
            \node[vertex] (\i) at (c.corner \i) {};
        \draw (1) to (6);
        \draw (2) to (15);
        \draw (3) to (8);
        \draw (4) to (17);
        \draw (5) to (11);
        \draw (7) to (13);
        \draw (9) to (19);
        \draw (10) to (16);
        \draw (12) to (18);
	\end{tikzpicture}\qquad
	\begin{tikzpicture}[]
		\newcommand{\scale}{4.5}	
		\node[regular polygon, regular polygon sides=27, minimum size=\scale*1cm, draw, black] at (0,0) (c) {}; 
        \foreach \i in {1,2,...,27}
            \node[vertex] (\i) at (c.corner \i) {};
        \draw (1) to (14);
        \draw (2) to (18);
        \draw (3) to (22);
	    \draw (4) to (11);
	    \draw (5) to (15);
	    \draw (6) to (26);
	    \draw (7) to (19);
	    \draw (8) to (23);
	    \draw (9) to (16);
	    \draw (10) to (27);
	    \draw (12) to (20);
	    \draw (13) to (24);
	    \draw (17) to (25);
	\end{tikzpicture}

	\caption{Smallest HIST-critical graphs of girth $4, 5, 6$ and $7$, respectively.
	}\label{fig:HIST-critical_girth_restrictions}
\end{figure}

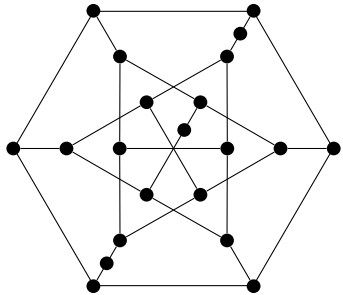
\begin{figure}[!htb]
	\centering
	\begin{tikzpicture}[]
		\newcommand{\scale}{1.4}	
		\node[regular polygon, regular polygon sides=6, minimum size=\scale*1cm] at (0,0) (ci) {}; 
        \foreach \i in {1,2,...,6}
            \node[vertex] (i\i) at (ci.corner \i) {};
        \node[regular polygon, regular polygon sides=6, minimum size=\scale*2cm] at (0,0) (cm) {}; 
        \foreach \i in {1,2,...,6}
            \node[vertex] (m\i) at (cm.corner \i) {};
        \node[regular polygon, regular polygon sides=6, minimum size=\scale*3cm, draw, black] at (0,0) (co) {}; 
        \foreach \i in {1,2,...,6}
            \node[vertex] (o\i) at (co.corner \i) {};

        \foreach \i in {1,2,...,6}
        	\draw (o\i) to (m\i);
    	\draw (m1) to (i2) to (m3) to (i4) to (m5) to (i6) to (m1);
    	\draw (i1) to (m2) to (i3) to (m4) to (i5) to (m6) to (i1);
    	\draw (i1) to (i4);
    	\draw (i2) to (i5);
    	\draw (i3) to (i6);

        \node[vertex] at ($(o1)!0.5!(m1)$) {}; 
        \node[vertex] at ($(i1)!0.3!(i4)$) {}; 
        \node[vertex] at ($(o4)!0.5!(m4)$) {}; 
	\end{tikzpicture}
	\caption{A HIST-critical graph of girth $6$. It is the Pappus graph with three edges subdivided.}\label{fig:pappus}
\end{figure}

\clearpage
\subsection{Certificates for the proof of Proposition~\ref{fragments}.}

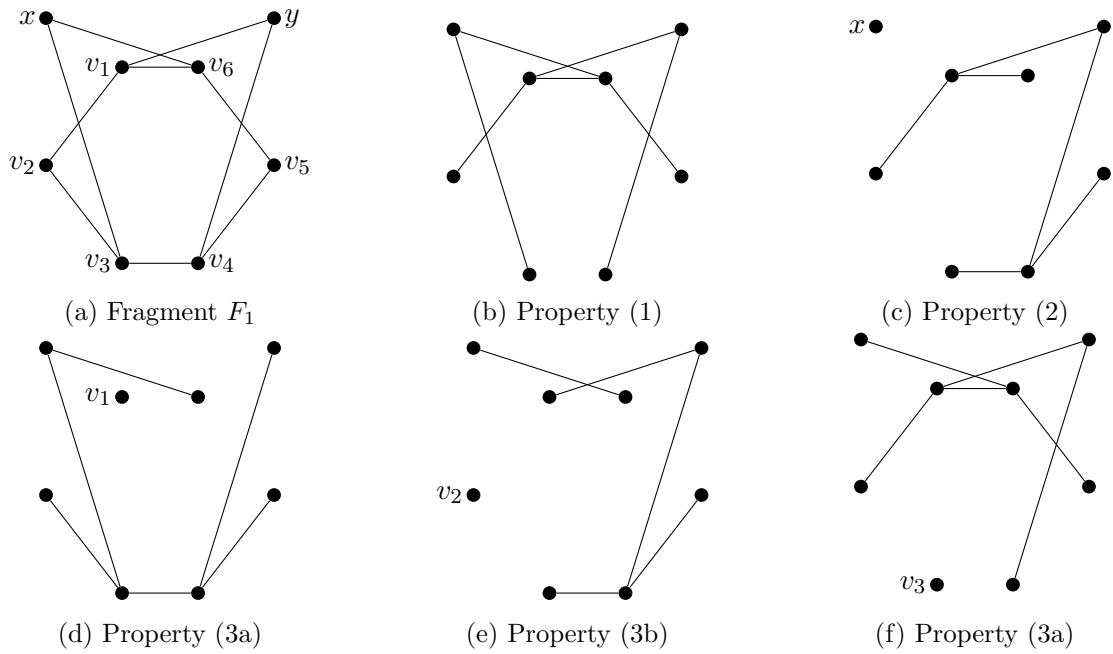
\begin{figure}[!htb]
    \centering
    \newcommand{\y}{1.3}
    \begin{subfigure}[t]{0.32\linewidth}
    \centering
    \begin{tikzpicture}[scale = 0.5]
        \node[vertex] (x) at (-3, 5*\y) {};
        \node[vertex] (y) at (3, 5*\y) {};
        \node[vertex] (1) at (-1, 4*\y) {};
        \node[vertex] (2) at (-3, 2*\y) {};
        \node[vertex] (3) at (-1, 0*\y) {};
        \node[vertex] (4) at (1, 0*\y) {};
        \node[vertex] (5) at (3, 2*\y) {};
        \node[vertex] (6) at (1, 4*\y) {};

         \path[draw]
        (x) edge node {} (3) 
        (x) edge node {} (6) 
        (y) edge node {} (1) 
        (y) edge node {} (4) 
        (1) edge node {} (2) 
        (2) edge node {} (3) 
        (3) edge node {} (4) 
        (4) edge node {} (5) 
        (5) edge node {} (6) 
        (6) edge node {} (1) 
        ;

\coordinate [label=left: {$x$}] () at (x);
\coordinate [label=right: {$y$}] () at (y);
\coordinate [label=left: {$v_1$}] () at (1);
\coordinate [label=left: {$v_2$}] () at (2);
\coordinate [label=left: {$v_3$}] () at (3);
\coordinate [label=right: {$v_4$}] () at (4);
\coordinate [label=right: {$v_5$}] () at (5);
\coordinate [label=right: {$v_6$}] () at (6);
    \end{tikzpicture}
    \caption{Fragment $F_1$}
    \end{subfigure}
    \begin{subfigure}[t]{0.32\linewidth}
    \centering
    \begin{tikzpicture}[scale = 0.5]
        \node[vertex] (x) at (-3, 5*\y) {};
        \node[vertex] (y) at (3, 5*\y) {};
        \node[vertex] (1) at (-1, 4*\y) {};
        \node[vertex] (2) at (-3, 2*\y) {};
        \node[vertex] (3) at (-1, 0*\y) {};
        \node[vertex] (4) at (1, 0*\y) {};
        \node[vertex] (5) at (3, 2*\y) {};
        \node[vertex] (6) at (1, 4*\y) {};

         \path[draw]
        (x) edge node {} (3) 
        (x) edge node {} (6) 
        (y) edge node {} (1) 
        (y) edge node {} (4) 
        (1) edge node {} (2) 
        (5) edge node {} (6) 
        (6) edge node {} (1) 
        ;

    \end{tikzpicture}
    \caption{Property (1)}
    \end{subfigure}
    \begin{subfigure}[t]{0.32\linewidth}
    \centering
    \begin{tikzpicture}[scale = 0.5]
        \node[vertex] (x) at (-3, 5*\y) {};
        \node[vertex] (y) at (3, 5*\y) {};
        \node[vertex] (1) at (-1, 4*\y) {};
        \node[vertex] (2) at (-3, 2*\y) {};
        \node[vertex] (3) at (-1, 0*\y) {};
        \node[vertex] (4) at (1, 0*\y) {};
        \node[vertex] (5) at (3, 2*\y) {};
        \node[vertex] (6) at (1, 4*\y) {};

         \path[draw]
        (y) edge node {} (1) 
        (y) edge node {} (4) 
        (1) edge node {} (2) 
        (3) edge node {} (4) 
        (4) edge node {} (5) 
        (6) edge node {} (1) 
        ;

\coordinate [label=left: {$x$}] () at (x);
    \end{tikzpicture}
    \caption{Property (2)}
    \end{subfigure}
    \begin{subfigure}[t]{0.32\linewidth}
    \centering

    \begin{tikzpicture}[scale = 0.5]
        \node[vertex] (x) at (-3, 5*\y) {};
        \node[vertex] (y) at (3, 5*\y) {};
        \node[vertex] (1) at (-1, 4*\y) {};
        \node[vertex] (2) at (-3, 2*\y) {};
        \node[vertex] (3) at (-1, 0*\y) {};
        \node[vertex] (4) at (1, 0*\y) {};
        \node[vertex] (5) at (3, 2*\y) {};
        \node[vertex] (6) at (1, 4*\y) {};

         \path[draw]
        (x) edge node {} (3) 
        (x) edge node {} (6) 
        (y) edge node {} (4) 
        (2) edge node {} (3) 
        (3) edge node {} (4) 
        (4) edge node {} (5) 
        ;

\coordinate [label=left: {$v_1$}] () at (1);
    \end{tikzpicture}
    \caption{Property (3a)}
    \end{subfigure}
    \begin{subfigure}[t]{0.32\linewidth}
    \centering
    \begin{tikzpicture}[scale = 0.5]
        \node[vertex] (x) at (-3, 5*\y) {};
        \node[vertex] (y) at (3, 5*\y) {};
        \node[vertex] (1) at (-1, 4*\y) {};
        \node[vertex] (2) at (-3, 2*\y) {};
        \node[vertex] (3) at (-1, 0*\y) {};
        \node[vertex] (4) at (1, 0*\y) {};
        \node[vertex] (5) at (3, 2*\y) {};
        \node[vertex] (6) at (1, 4*\y) {};
         \path[draw]
        (x) edge node {} (6) 
        (y) edge node {} (1) 
        (y) edge node {} (4) 
        (3) edge node {} (4) 
        (4) edge node {} (5) 
        ;

\coordinate [label=left: {$v_2$}] () at (2);
    \end{tikzpicture}
    \caption{Property (3b)}
    \end{subfigure}
    \begin{subfigure}[t]{0.32\linewidth}
    \centering
    \begin{tikzpicture}[scale = 0.5]
        \node[vertex] (x) at (-3, 5*\y) {};
        \node[vertex] (y) at (3, 5*\y) {};
        \node[vertex] (1) at (-1, 4*\y) {};
        \node[vertex] (2) at (-3, 2*\y) {};
        \node[vertex] (3) at (-1, 0*\y) {};
        \node[vertex] (4) at (1, 0*\y) {};
        \node[vertex] (5) at (3, 2*\y) {};
        \node[vertex] (6) at (1, 4*\y) {};

         \path[draw]
        (x) edge node {} (6) 
        (y) edge node {} (1) 
        (y) edge node {} (4) 
        (1) edge node {} (2) 
        (5) edge node {} (6) 
        (6) edge node {} (1) 
        ;

\coordinate [label=left: {$v_3$}] () at (3);
    \end{tikzpicture}
    \caption{Property (3a)}
    \end{subfigure}

    \caption{Fragment $F_1$ with properties (2) and (3) defined in Section 3.1.}\label{fig:3-1}
\end{figure}

\begin{figure}[!htb]
    \centering
    \begin{subfigure}[t]{0.32\linewidth}
    \centering
    \begin{tikzpicture}[scale = 0.5]
        \node[vertex] (x) at (-3, 7) {};
        \node[vertex] (y) at (3, 7) {};
        \node[vertex] (1) at (-1, 6) {};
        \node[vertex] (2) at (-3, 3) {};
        \node[vertex] (3) at (-1, 0) {};
        \node[vertex] (4) at (1, 0) {};
        \node[vertex] (5) at (3, 3) {};
        \node[vertex] (6) at (1, 6) {};
        \node[vertex] (7) at (-1, 4) {};
        \node[vertex] (8) at (-1, 2) {};
        \node[vertex] (9) at (1, 2) {};
        \node[vertex] (10) at (1, 4) {};

         \path[draw]
        (x) edge node {} (1) 
        (x) edge node {} (8) 
        (y) edge node {} (6) 
        (y) edge node {} (9) 
        (1) edge node {} (2) 
        (2) edge node {} (3) 
        (3) edge node {} (4) 
        (4) edge node {} (5) 
        (5) edge node {} (6) 
        (6) edge node {} (1) 
        (1) edge node {} (7) 
        (3) edge node {} (8) 
        (4) edge node {} (9) 
        (6) edge node {} (10) 
        (7) edge node {} (9) 
        (8) edge node {} (10) 
        ;

\coordinate [label=left: {$x$}] () at (x);
\coordinate [label=right: {$y$}] () at (y);
\coordinate [label=left: {$v_1$}] () at (1);
\coordinate [label=left: {$v_2$}] () at (2);
\coordinate [label=left: {$v_3$}] () at (3);
\coordinate [label=right: {$v_4$}] () at (4);
\coordinate [label=right: {$v_5$}] () at (5);
\coordinate [label=right: {$v_6$}] () at (6);
\coordinate [label=left: {$v_7$}] () at (7);
\coordinate [label=left: {$v_8$}] () at (8);
\coordinate [label=right: {$v_9$}] () at (9);
\coordinate [label=left: {$v_{10}$}] () at (10);
    \end{tikzpicture}
    \caption{Fragment $F_2$}
    \end{subfigure}
    \begin{subfigure}[t]{0.32\linewidth}
    \centering
    \begin{tikzpicture}[scale = 0.5]
        \node[vertex] (x) at (-3, 7) {};
        \node[vertex] (y) at (3, 7) {};
        \node[vertex] (1) at (-1, 6) {};
        \node[vertex] (2) at (-3, 3) {};
        \node[vertex] (3) at (-1, 0) {};
        \node[vertex] (4) at (1, 0) {};
        \node[vertex] (5) at (3, 3) {};
        \node[vertex] (6) at (1, 6) {};
        \node[vertex] (7) at (-1, 4) {};
        \node[vertex] (8) at (-1, 2) {};
        \node[vertex] (9) at (1, 2) {};
        \node[vertex] (10) at (1, 4) {};

         \path[draw]
        (x) edge node {} (1) 
        (x) edge node {} (8) 
        (y) edge node {} (6) 
        (y) edge node {} (9) 
        (1) edge node {} (2) 
        (5) edge node {} (6) 
        (6) edge node {} (1) 
        (3) edge node {} (8) 
        (4) edge node {} (9) 
        (7) edge node {} (9) 
        (8) edge node {} (10) 
        ;

    \end{tikzpicture}
    \caption{Property (1)}
    \end{subfigure}
    \begin{subfigure}[t]{0.32\linewidth}
    \centering
    \begin{tikzpicture}[scale = 0.5]
        \node[vertex] (x) at (-3, 7) {};
        \node[vertex] (y) at (3, 7) {};
        \node[vertex] (1) at (-1, 6) {};
        \node[vertex] (2) at (-3, 3) {};
        \node[vertex] (3) at (-1, 0) {};
        \node[vertex] (4) at (1, 0) {};
        \node[vertex] (5) at (3, 3) {};
        \node[vertex] (6) at (1, 6) {};
        \node[vertex] (7) at (-1, 4) {};
        \node[vertex] (8) at (-1, 2) {};
        \node[vertex] (9) at (1, 2) {};
        \node[vertex] (10) at (1, 4) {};

         \path[draw]
        (y) edge node {} (6) 
        (y) edge node {} (9) 
        (2) edge node {} (3) 
        (3) edge node {} (4) 
        (4) edge node {} (5) 
        (6) edge node {} (1) 
        (3) edge node {} (8) 
        (4) edge node {} (9) 
        (6) edge node {} (10) 
        (7) edge node {} (9) 
        ;

\coordinate [label=left: {$x$}] () at (x);
    \end{tikzpicture}
    \caption{Property (2)}
    \end{subfigure}
    \begin{subfigure}[t]{0.32\linewidth}
    \centering

    \begin{tikzpicture}[scale = 0.5]
        \node[vertex] (x) at (-3, 7) {};
        \node[vertex] (y) at (3, 7) {};
        \node[vertex] (1) at (-1, 6) {};
        \node[vertex] (2) at (-3, 3) {};
        \node[vertex] (3) at (-1, 0) {};
        \node[vertex] (4) at (1, 0) {};
        \node[vertex] (5) at (3, 3) {};
        \node[vertex] (6) at (1, 6) {};
        \node[vertex] (7) at (-1, 4) {};
        \node[vertex] (8) at (-1, 2) {};
        \node[vertex] (9) at (1, 2) {};
        \node[vertex] (10) at (1, 4) {};

         \path[draw]
        (x) edge node {} (8) 
        (y) edge node {} (6) 
        (y) edge node {} (9) 
        (2) edge node {} (3) 
        (3) edge node {} (4) 
        (4) edge node {} (5) 
        (3) edge node {} (8) 
        (4) edge node {} (9) 
        (7) edge node {} (9) 
        (8) edge node {} (10) 
        ;

\coordinate [label=left: {$v_1$}] () at (1);
    \end{tikzpicture}\qquad
    \caption{Property (3a)}
    \end{subfigure}
    \begin{subfigure}[t]{0.32\linewidth}
    \centering
    \begin{tikzpicture}[scale = 0.5]
        \node[vertex] (x) at (-3, 7) {};
        \node[vertex] (y) at (3, 7) {};
        \node[vertex] (1) at (-1, 6) {};
        \node[vertex] (2) at (-3, 3) {};
        \node[vertex] (3) at (-1, 0) {};
        \node[vertex] (4) at (1, 0) {};
        \node[vertex] (5) at (3, 3) {};
        \node[vertex] (6) at (1, 6) {};
        \node[vertex] (7) at (-1, 4) {};
        \node[vertex] (8) at (-1, 2) {};
        \node[vertex] (9) at (1, 2) {};
        \node[vertex] (10) at (1, 4) {};

         \path[draw]
        (x) edge node {} (8) 
        (y) edge node {} (6) 
        (y) edge node {} (9) 
        (3) edge node {} (4) 
        (4) edge node {} (5) 
        (6) edge node {} (1) 
        (4) edge node {} (9) 
        (6) edge node {} (10) 
        (7) edge node {} (9) 
        ;

\coordinate [label=left: {$v_2$}] () at (2);
    \end{tikzpicture}\qquad
    \caption{Property (3b)}
    \end{subfigure}
    \begin{subfigure}[t]{0.32\linewidth}
    \centering
    \begin{tikzpicture}[scale = 0.5]
        \node[vertex] (x) at (-3, 7) {};
        \node[vertex] (y) at (3, 7) {};
        \node[vertex] (1) at (-1, 6) {};
        \node[vertex] (2) at (-3, 3) {};
        \node[vertex] (3) at (-1, 0) {};
        \node[vertex] (4) at (1, 0) {};
        \node[vertex] (5) at (3, 3) {};
        \node[vertex] (6) at (1, 6) {};
        \node[vertex] (7) at (-1, 4) {};
        \node[vertex] (8) at (-1, 2) {};
        \node[vertex] (9) at (1, 2) {};
        \node[vertex] (10) at (1, 4) {};

         \path[draw]
        (x) edge node {} (1) 
        (x) edge node {} (8) 
        (y) edge node {} (9) 
        (1) edge node {} (2) 
        (5) edge node {} (6) 
        (6) edge node {} (1) 
        (4) edge node {} (9) 
        (6) edge node {} (10) 
        (7) edge node {} (9) 
        ;

\coordinate [label=left: {$v_3$}] () at (3);
    \end{tikzpicture}
    \caption{Property (3b)}
    \end{subfigure}
    \begin{subfigure}[t]{0.32\linewidth}
    \centering
    \begin{tikzpicture}[scale = 0.5]
        \node[vertex] (x) at (-3, 7) {};
        \node[vertex] (y) at (3, 7) {};
        \node[vertex] (1) at (-1, 6) {};
        \node[vertex] (2) at (-3, 3) {};
        \node[vertex] (3) at (-1, 0) {};
        \node[vertex] (4) at (1, 0) {};
        \node[vertex] (5) at (3, 3) {};
        \node[vertex] (6) at (1, 6) {};
        \node[vertex] (7) at (-1, 4) {};
        \node[vertex] (8) at (-1, 2) {};
        \node[vertex] (9) at (1, 2) {};
        \node[vertex] (10) at (1, 4) {};

         \path[draw]
        (x) edge node {} (8) 
        (y) edge node {} (6) 
        (y) edge node {} (9) 
        (2) edge node {} (3) 
        (3) edge node {} (4) 
        (5) edge node {} (6) 
        (6) edge node {} (1) 
        (3) edge node {} (8) 
        (8) edge node {} (10) 
        ;

\coordinate [label=left: {$v_7$}] () at (7);
    \end{tikzpicture}\qquad
    \caption{Property (3b)}
    \end{subfigure}
    \begin{subfigure}[t]{0.32\linewidth}
    \centering
    \begin{tikzpicture}[scale = 0.5]
        \node[vertex] (x) at (-3, 7) {};
        \node[vertex] (y) at (3, 7) {};
        \node[vertex] (1) at (-1, 6) {};
        \node[vertex] (2) at (-3, 3) {};
        \node[vertex] (3) at (-1, 0) {};
        \node[vertex] (4) at (1, 0) {};
        \node[vertex] (5) at (3, 3) {};
        \node[vertex] (6) at (1, 6) {};
        \node[vertex] (7) at (-1, 4) {};
        \node[vertex] (8) at (-1, 2) {};
        \node[vertex] (9) at (1, 2) {};
        \node[vertex] (10) at (1, 4) {};

         \path[draw]
        (x) edge node {} (1) 
        (y) edge node {} (6) 
        (y) edge node {} (9) 
        (1) edge node {} (2) 
        (3) edge node {} (4) 
        (4) edge node {} (5) 
        (6) edge node {} (1) 
        (4) edge node {} (9) 
        (6) edge node {} (10) 
        (7) edge node {} (9) 
        ;

\coordinate [label=left: {$v_8$}] () at (8);
    \end{tikzpicture}
    \caption{Property (3a)}
    \end{subfigure}

    \caption{Fragment $F_2$ with properties (2) and (3) defined in Section 3.1.}\label{fig:3-2}
\end{figure}

\clearpage

\subsection{Number of planar \texorpdfstring{$\bm4$}{4}-connected graphs}\label{app:4-conn_planar}

\begin{table}[!htb]
\centering
\begin{tabular}{c | r}
    Order & \# $4$-conn. planar\\\hline
    6 & 1\\
    7 & 1\\
    8 & 4\\
    9 & 10\\
    10 & 53\\
    11 & 292\\
    12 & 2\,224\\
    13 & 18\,493\\
    14 & 167\,504\\
    15 & 1\,571\,020\\
    16 & 15\,151\,289\\
    17 & 148\,864\,939\\
    18 & 1\,485\,904\,672\\
    19 & 15\,028\,654\,628\\
    20 & 153\,781\,899\,708\\
    21 & 1\,589\,921\,572\,902\\
    22 & 16\,591\,187\,039\,082\\
\end{tabular}
\caption{The number of planar $4$-connected graphs for each given order.}\label{tab:4-conn_planar} 
\end{table}

\clearpage
\subsection{Planar 4-connected graphs with fewest number of HISTs}

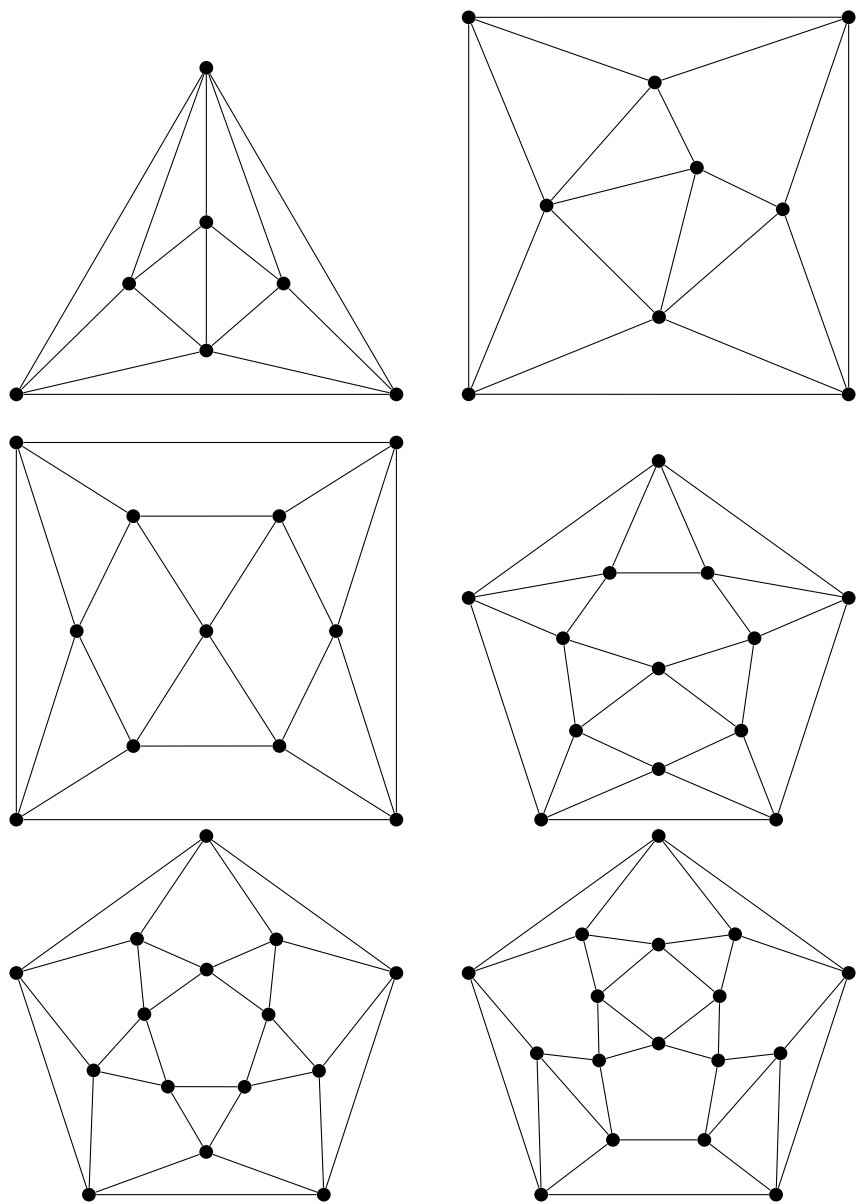
\begin{figure}[!htb]
    \centering

    \begin{tikzpicture}[scale=0.05]
        \definecolor{marked}{rgb}{0.25,0.5,0.25}
        \node [circle,fill,scale=0.5] (7) at (49.981223,18.343974) {};
        \node [circle,fill,scale=0.5] (6) at (99.999999,6.702966) {};
        \node [circle,fill,scale=0.5] (5) at (70.296657,36.105896) {};
        \node [circle,fill,scale=0.5] (4) at (49.981223,52.365753) {};
        \node [circle,fill,scale=0.5] (3) at (29.665789,36.068345) {};
        \node [circle,fill,scale=0.5] (2) at (0.000000,6.702966) {};
        \node [circle,fill,scale=0.5] (1) at (49.981223,93.297033) {};
        \draw [black] (7) to (2);
        \draw [black] (7) to (3);
        \draw [black] (7) to (4);
        \draw [black] (7) to (5);
        \draw [black] (7) to (6);
        \draw [black] (6) to (1);
        \draw [black] (6) to (2);
        \draw [black] (6) to (5);
        \draw [black] (5) to (1);
        \draw [black] (5) to (4);
        \draw [black] (4) to (1);
        \draw [black] (4) to (3);
        \draw [black] (3) to (1);
        \draw [black] (3) to (2);
        \draw [black] (2) to (1);
    \end{tikzpicture}\qquad
    \begin{tikzpicture}[scale=0.05]
        \definecolor{marked}{rgb}{0.25,0.5,0.25}
        \node [circle,fill,scale=0.5] (9) at (48.973395,82.706766) {};
        \node [circle,fill,scale=0.5] (8) at (60.049162,60.121457) {};
        \node [circle,fill,scale=0.5] (7) at (20.517641,50.086754) {};
        \node [circle,fill,scale=0.5] (6) at (82.663388,49.074609) {};
        \node [circle,fill,scale=0.5] (5) at (50.101216,20.503180) {};
        \node [circle,fill,scale=0.5] (4) at (99.985541,99.999999) {};
        \node [circle,fill,scale=0.5] (3) at (0.014458,99.971082) {};
        \node [circle,fill,scale=0.5] (2) at (0.014459,0.028917) {};
        \node [circle,fill,scale=0.5] (1) at (99.985540,0.000000) {};
        \draw [black] (9) to (3);
        \draw [black] (9) to (4);
        \draw [black] (9) to (8);
        \draw [black] (9) to (7);
        \draw [black] (8) to (6);
        \draw [black] (8) to (5);
        \draw [black] (8) to (7);
        \draw [black] (7) to (2);
        \draw [black] (7) to (3);
        \draw [black] (7) to (5);
        \draw [black] (6) to (1);
        \draw [black] (6) to (5);
        \draw [black] (6) to (4);
        \draw [black] (5) to (1);
        \draw [black] (5) to (2);
        \draw [black] (4) to (1);
        \draw [black] (4) to (3);
        \draw [black] (3) to (2);
        \draw [black] (2) to (1);
    \end{tikzpicture}\bigskip

    \begin{tikzpicture}[scale=0.05]
        \definecolor{marked}{rgb}{0.25,0.5,0.25}
        \node [circle,fill,scale=0.5] (11) at (69.228805,19.535239) {};
        \node [circle,fill,scale=0.5] (10) at (69.203268,80.490296) {};
        \node [circle,fill,scale=0.5] (9) at (84.116445,50.000000) {};
        \node [circle,fill,scale=0.5] (8) at (50.000000,50.000000) {};
        \node [circle,fill,scale=0.5] (7) at (30.771194,80.464760) {};
        \node [circle,fill,scale=0.5] (6) at (30.796731,19.509703) {};
        \node [circle,fill,scale=0.5] (5) at (15.883554,50.000000) {};
        \node [circle,fill,scale=0.5] (4) at (99.999999,0.000000) {};
        \node [circle,fill,scale=0.5] (3) at (99.999999,99.999999) {};
        \node [circle,fill,scale=0.5] (2) at (0.000000,99.999999) {};
        \node [circle,fill,scale=0.5] (1) at (0.000000,0.000000) {};
        \draw [black] (11) to (6);
        \draw [black] (11) to (8);
        \draw [black] (11) to (9);
        \draw [black] (11) to (4);
        \draw [black] (10) to (3);
        \draw [black] (10) to (9);
        \draw [black] (10) to (8);
        \draw [black] (10) to (7);
        \draw [black] (9) to (3);
        \draw [black] (9) to (4);
        \draw [black] (8) to (6);
        \draw [black] (8) to (7);
        \draw [black] (7) to (2);
        \draw [black] (7) to (5);
        \draw [black] (6) to (1);
        \draw [black] (6) to (5);
        \draw [black] (5) to (1);
        \draw [black] (5) to (2);
        \draw [black] (4) to (1);
        \draw [black] (4) to (3);
        \draw [black] (3) to (2);
        \draw [black] (2) to (1);
    \end{tikzpicture}\qquad
    \begin{tikzpicture}[scale=0.05]
        \definecolor{marked}{rgb}{0.25,0.5,0.25}
        \node [circle,fill,scale=0.5] (13) at (50.000000,15.859953) {};
        \node [circle,fill,scale=0.5] (12) at (28.255529,26.031944) {};
        \node [circle,fill,scale=0.5] (11) at (71.744470,26.056513) {};
        \node [circle,fill,scale=0.5] (10) at (50.000000,42.518429) {};
        \node [circle,fill,scale=0.5] (9) at (24.815726,50.552827) {};
        \node [circle,fill,scale=0.5] (8) at (75.208844,50.552827) {};
        \node [circle,fill,scale=0.5] (7) at (62.874691,67.874693) {};
        \node [circle,fill,scale=0.5] (6) at (37.125308,67.874693) {};
        \node [circle,fill,scale=0.5] (5) at (0.000000,61.240786) {};
        \node [circle,fill,scale=0.5] (4) at (19.090910,2.444720) {};
        \node [circle,fill,scale=0.5] (3) at (80.909089,2.444720) {};
        \node [circle,fill,scale=0.5] (2) at (99.999999,61.240786) {};
        \node [circle,fill,scale=0.5] (1) at (50.000000,97.555279) {};
        \draw [black] (13) to (3);
        \draw [black] (13) to (4);
        \draw [black] (13) to (12);
        \draw [black] (13) to (11);
        \draw [black] (12) to (4);
        \draw [black] (12) to (9);
        \draw [black] (12) to (10);
        \draw [black] (11) to (3);
        \draw [black] (11) to (10);
        \draw [black] (11) to (8);
        \draw [black] (10) to (8);
        \draw [black] (10) to (9);
        \draw [black] (9) to (5);
        \draw [black] (9) to (6);
        \draw [black] (8) to (2);
        \draw [black] (8) to (7);
        \draw [black] (7) to (1);
        \draw [black] (7) to (2);
        \draw [black] (7) to (6);
        \draw [black] (6) to (1);
        \draw [black] (6) to (5);
        \draw [black] (5) to (1);
        \draw [black] (5) to (4);
        \draw [black] (4) to (3);
        \draw [black] (3) to (2);
        \draw [black] (2) to (1);
    \end{tikzpicture}
    \begin{tikzpicture}[scale=0.05]
        \definecolor{marked}{rgb}{0.25,0.5,0.25}
        \node [circle,fill,scale=0.5] (15) at (20.316025,35.393553) {};
        \node [circle,fill,scale=0.5] (14) at (33.691496,50.322439) {};
        \node [circle,fill,scale=0.5] (13) at (39.868906,31.093184) {};
        \node [circle,fill,scale=0.5] (12) at (49.965178,13.772220) {};
        \node [circle,fill,scale=0.5] (11) at (19.096909,2.449041) {};
        \node [circle,fill,scale=0.5] (10) at (0.000000,61.223283) {};
        \node [circle,fill,scale=0.5] (9) at (31.749272,70.269295) {};
        \node [circle,fill,scale=0.5] (8) at (50.077707,62.137506) {};
        \node [circle,fill,scale=0.5] (7) at (66.368600,50.215836) {};
        \node [circle,fill,scale=0.5] (6) at (60.062367,31.045984) {};
        \node [circle,fill,scale=0.5] (5) at (80.901999,2.446682) {};
        \node [circle,fill,scale=0.5] (4) at (49.999980,97.553317) {};
        \node [circle,fill,scale=0.5] (3) at (68.408378,70.139536) {};
        \node [circle,fill,scale=0.5] (2) at (79.644587,35.268528) {};
        \node [circle,fill,scale=0.5] (1) at (99.999999,61.213892) {};
        \draw [black] (15) to (11);
        \draw [black] (15) to (10);
        \draw [black] (15) to (14);
        \draw [black] (15) to (13);
        \draw [black] (14) to (9);
        \draw [black] (14) to (8);
        \draw [black] (14) to (13);
        \draw [black] (13) to (6);
        \draw [black] (13) to (12);
        \draw [black] (12) to (5);
        \draw [black] (12) to (11);
        \draw [black] (12) to (6);
        \draw [black] (11) to (10);
        \draw [black] (11) to (5);
        \draw [black] (10) to (4);
        \draw [black] (10) to (9);
        \draw [black] (9) to (4);
        \draw [black] (9) to (8);
        \draw [black] (8) to (3);
        \draw [black] (8) to (7);
        \draw [black] (7) to (2);
        \draw [black] (7) to (6);
        \draw [black] (7) to (3);
        \draw [black] (6) to (2);
        \draw [black] (5) to (1);
        \draw [black] (5) to (2);
        \draw [black] (4) to (1);
        \draw [black] (4) to (3);
        \draw [black] (3) to (1);
        \draw [black] (2) to (1);
    \end{tikzpicture}\qquad
    \begin{tikzpicture}[scale=0.05]
        \definecolor{marked}{rgb}{0.25,0.5,0.25}
        \node [circle,fill,scale=0.5] (17) at (33.921715,55.084041) {};
        \node [circle,fill,scale=0.5] (16) at (34.323418,38.032596) {};
        \node [circle,fill,scale=0.5] (15) at (49.989518,42.524167) {};
        \node [circle,fill,scale=0.5] (14) at (66.060042,55.084153) {};
        \node [circle,fill,scale=0.5] (13) at (49.988132,68.759127) {};
        \node [circle,fill,scale=0.5] (12) at (29.861993,71.479931) {};
        \node [circle,fill,scale=0.5] (11) at (0.000000,61.224779) {};
        \node [circle,fill,scale=0.5] (10) at (17.955129,39.947373) {};
        \node [circle,fill,scale=0.5] (9) at (37.975701,16.972358) {};
        \node [circle,fill,scale=0.5] (8) at (65.679167,38.036175) {};
        \node [circle,fill,scale=0.5] (7) at (70.121524,71.504002) {};
        \node [circle,fill,scale=0.5] (6) at (50.000343,97.552155) {};
        \node [circle,fill,scale=0.5] (5) at (19.094069,2.447988) {};
        \node [circle,fill,scale=0.5] (4) at (62.026863,16.973241) {};
        \node [circle,fill,scale=0.5] (3) at (82.050829,39.946135) {};
        \node [circle,fill,scale=0.5] (2) at (99.999999,61.229587) {};
        \node [circle,fill,scale=0.5] (1) at (80.900805,2.447844) {};
        \draw [black] (17) to (15);
        \draw [black] (17) to (16);
        \draw [black] (17) to (12);
        \draw [black] (17) to (13);
        \draw [black] (16) to (9);
        \draw [black] (16) to (10);
        \draw [black] (16) to (15);
        \draw [black] (15) to (14);
        \draw [black] (15) to (8);
        \draw [black] (14) to (8);
        \draw [black] (14) to (13);
        \draw [black] (14) to (7);
        \draw [black] (13) to (12);
        \draw [black] (13) to (7);
        \draw [black] (12) to (11);
        \draw [black] (12) to (6);
        \draw [black] (11) to (6);
        \draw [black] (11) to (10);
        \draw [black] (11) to (5);
        \draw [black] (10) to (9);
        \draw [black] (10) to (5);
        \draw [black] (9) to (5);
        \draw [black] (9) to (4);
        \draw [black] (8) to (3);
        \draw [black] (8) to (4);
        \draw [black] (7) to (2);
        \draw [black] (7) to (6);
        \draw [black] (6) to (2);
        \draw [black] (5) to (1);
        \draw [black] (4) to (1);
        \draw [black] (4) to (3);
        \draw [black] (3) to (1);
        \draw [black] (3) to (2);
        \draw [black] (2) to (1);
    \end{tikzpicture}
 
    \caption{The planar $4$-connected graphs of odd order $n$ attaining the minimum number of HISTs for each order up to order $17$.}\label{fig:4-conn_planar_attaining_min_HISTs}
\end{figure}

\end{document}